\documentclass[3p,lefttitle]{elsarticle}
\pdfoutput=1
\usepackage{amsmath}
\usepackage{amssymb}
\usepackage{amsthm}
\usepackage{thmtools} 
\usepackage{mathtools}
\usepackage{enumerate}

\usepackage[colorlinks=true]{hyperref}
\usepackage[noabbrev,capitalize]{cleveref}

\renewcommand{\cref}{\Cref}

\NewCommandCopy{\saveciteauthor}{\citeauthor}
\renewcommand{\citeauthor}[1]{\begin{NoHyper}\saveciteauthor{#1}\end{NoHyper}}

\usepackage{xspace}
\usepackage{tikz-cd}
\tikzcdset{arrow style=math font}

\usepackage{stmaryrd}
\usepackage{microtype}

\renewcommand{\eqref}[1]{(\ref{#1})}

\newcommand*{\comp}{\mathrel{\circ}}
\newcommand*{\axioms}{\operatorname{\mathfrak{a}}}
\newcommand*{\axiomspi}{\axioms_{\Pi}}
\newcommand*{\SA}{S_{\axioms}}
\newcommand*{\rhoa}{\rho_{\axioms}}
\newcommand*{\piproj}{\operatorname{pr}}
\newcommand*{\sigmainj}{\operatorname{in}}
\newcommand*{\treq}{\operatorname{treq}}

\newcommand*{\eqtoid}{\operatorname{eqtoid}}
\newcommand*{\eqtoidt}[3]{\eqtoid_{{#1},{#2}}{#3}}
\newcommand*{\R}{\mathbb R}
\newcommand*{\Nat}{\mathbb N}
\newcommand*{\NI}{\mathbb N_\infty}
\newcommand*{\Zero}{\mathbf 0}
\newcommand*{\One}{\mathbf 1}
\newcommand*{\Two}{\mathbf 2}

\newcommand*{\UU}{\mathcal U}
\newcommand*{\VV}{\mathcal V}
\newcommand*{\WW}{\mathcal W}
\newcommand*{\TT}{\mathcal T}
\newcommand*{\MCU}{\mathcal U}

\newcommand*{\bM}{\mathbb M}
\DeclareMathOperator{\idtoeq}{idtoeqv}

\DeclareMathOperator{\ssup}{ssup}
\DeclareMathOperator{\iroot}{root}
\DeclareMathOperator{\forest}{forest}

\newcommand*{\bV}{\mathbb V}

\DeclareMathOperator{\decomp}{decomposition}
\DeclareMathOperator{\fiber}{fiber}
\DeclareMathOperator{\transport}{transport}

\newcommand{\transportd}[5]{\operatorname{transportd}^{#1}_{#2}{#3}\,{#4}\,{#5}}

\newcommand*{\transportt}[3]{\transport^{#1}{#2}\,{#3}}

\newcommand*{\extension}[3]{\operatorname{extension}{#1}\,{#2}\,{#3}}
\newcommand*{\extends}[3]{\operatorname{extends}{#1}\,{#2}\,{#3}}
\newcommand*{\extendsf}[4]{\extends{#1}{#2}{#3}\,{#4}}

\newcommand*{\rhotwo}[2]{\rho_{{#1},{#2}}}

\newcommand*{\rhosuper}[4]{\rho_{{#1},{#2}}^{{#3},{#4}}}

\newcommand*{\rhosigma}{\rho_{\Sigma}}

\DeclarePairedDelimiter{\pair}{\langle}{\rangle}
\DeclarePairedDelimiter{\paren}{(}{)}

\DeclareMathOperator{\fst}{pr_1}
\DeclareMathOperator{\snd}{pr_2}

\DeclareMathOperator{\inl}{inl}
\DeclareMathOperator{\inr}{inr}

\DeclareMathOperator{\refl}{refl}

\DeclareMathOperator{\id}{id}
\DeclareMathOperator{\lifting}{\mathcal L}
\DeclareMathOperator{\isprop}{is-prop}

\DeclareMathOperator{\modality}{\bigcirc}

\newcommand*{\funone}[2]{#1\,#2}
\newcommand*{\funtwo}[3]{#1\,#2\,#3}

\newcommand{\apart}{\mathrel{\#}}

\newcommand{\Inh}{\operatorname{Inh}}
\newcommand{\NE}{\operatorname{NE}}

\newcommand{\weakresizing}[1]{\Omega_{\neg\neg}\text{-resizing}_{#1}}

\DeclarePairedDelimiter{\proptrunc}{\|}{\|}

\DeclarePairedDelimiter{\pa}{(}{)}

\newcommand{\Omeganotnot}[1]{\Omega_{#1}^{\neg\neg}}

\newcommand*{\RPinf}{\mathbb{R}\mathsf{P}^\infty}

\DeclareMathOperator{\im}{im}

\newtheorem{theorem}{Theorem}[section]
\newtheorem{lemma}[theorem]{Lemma}
\newtheorem{proposition}[theorem]{Proposition}
\newtheorem{corollary}[theorem]{Corollary}
\theoremstyle{definition}
\newtheorem{definition}[theorem]{Definition}
\newtheorem{example}[theorem]{Example}
\newtheorem{examples}[theorem]{Examples}
\newtheorem{counterexample}[theorem]{Counterexample}
\theoremstyle{remark}
\newtheorem{remark}[theorem]{Remark}

\begin{document}

\title{Examples and counterexamples of injective types}

\author[1]{\texorpdfstring{Tom de Jong\corref{cor1}}{Tom de Jong}}
\ead{tom.dejong@nottingham.ac.uk}
\ead[url]{https://www.tdejong.com}

\author[2]{Mart\'in H\"otzel Escard\'o}
\ead{m.escardo@bham.ac.uk}
\ead[url]{https://www.cs.bham.ac.uk/~mhe}

\cortext[cor1]{Corresponding author}

\affiliation[1]{organization={School of Computer Science, University of Nottingham},
  city={Nottingham},
  country={United Kingdom}
}

\affiliation[2]{organization={School of Computer Science, University of Birmingham},
  city={Birmingham},
  country={United Kingdom}
}

\begin{keyword}
  injective and flabby types \sep
  homotopy type theory and univalent foundations \sep
  constructivity and predicativity \sep
  Rice's theorem \sep
  De Morgan Law and weak excluded middle \sep
  (tight) apartness relation
  \MSC[2020]{03B38,03F65,03B70}
\end{keyword}

\begin{abstract}

  It is known that, in univalent mathematics, type universes, the type
  of $n$-types in a universe, reflective subuniverses, and the
  underlying type of any algebra of the lifting monad are all
  (algebraically) injective.  Here, we further show that the type of
  ordinals, the type of iterative (multi)sets, the underlying type of
  any pointed directed complete poset, as well as the types of (small)
  \(\infty\)-magmas, monoids, and groups are all injective, among
  other examples.

  Not all types of mathematical structures are injective in
  general. For example, the type of inhabited types is
  injective if and only if all propositions are \emph{projective}.
  In contrast, the type of pointed types and the type of non-empty
  types are always injective.
  The injectivity of the type of two-element types implies
  Fourman and \v{S}\v{c}edrov's \emph{world's simplest axiom of choice}.
  We also show that there are no nontrivial \emph{small} injective types
  unless a weak propositional resizing principle holds.

  Other counterexamples include the type of
  booleans, the simple types, the type of Dedekind reals, and the type
  of conatural numbers, whose injectivity implies weak excluded middle.
  More generally, any type with an apartness relation and two points
  apart cannot be injective unless weak excluded middle holds.
  Finally, we show that injective types have no non-trivial decidable
  properties, unless weak excluded middle holds, which amounts to a
  Rice-like theorem for injective types.

\end{abstract}

\maketitle

\section{Introduction}

In the context of univalent mathematics~\cite{HoTTBook}, injective types were
discussed by the second author in~\cite{Escardo2021}. The interest in
injectivity originated in its use to construct infinite searchable
types~\cite{Escardo2019}, but the topic turned out to have a rather rich theory
on its own.
With {classical} logic, i.e.\ in the presence of excluded middle, the
(algebraically) injective types are precisely the pointed types.
In fact, this characterization is equivalent to excluded middle.
In constructive univalent mathematics, the situation is more interesting, as shown in~\cite{Escardo2021} and
further emphasized in this paper.

In~\cite{Escardo2021}, type universes, the type of propositions in a given
universe and the underlying type of any algebra of the lifting monad were all shown to be injective.
Here, we further show that the type of ordinals~\cite[\S10.3]{HoTTBook}, the
type of iterative (multi) sets~\cite{Gylterud2018,Gylterud2019}, the underlying
type of any pointed directed complete poset, as well as the types of (small)
\(\infty\)-magmas, monoids, and groups are all injective.
The injectivity of magmas, monoids, and groups, follows from a general result
(\cref{Sigma-flabby}) that gives a sufficient condition for the injectivity of a
\(\Sigma\)\nobreakdash-type.
More precisely, we use a specialization of this general condition to
study mathematical structures as given by \(\Sigma\)-types over type
universes.
For \emph{subtypes} of injective types, we present a necessary and sufficient
condition (\cref{theorem:injective:subtype}).

Complementary to the above positive results, we present several counterexamples
of injective types, although we should be careful in what we mean by
counterexample. The only type that is \emph{provably} not injective is the
empty type, because classically any pointed type is injective.
But we present several types that cannot be shown to be injective in
constructive mathematics, because their injectivity would imply a constructive
taboo, which is a statement that it is not constructively provable and false in
some (topos) models.
Often, the relevant taboo for us is \emph{weak excluded middle} which says that
the negation of any proposition is decidable, and which is equivalent to De
Morgan's Law~\cite[Proposition~D4.6.2]{elephant}. Indeed, weak excluded middle
holds as soon as any of the following types are injective: the type of booleans
\(\Two \coloneqq \One + \One\) (as already shown in \emph{loc.\ cit.}), the
simple types (obtained from \(\Nat\) by iterating function types), the type of
Dedekind reals, and the type of conatural numbers~\cite{Escardo2013}.
More generally, any type with an apartness relation and two points apart cannot
be injective unless weak excluded middle holds
(\cref{injective-type-with-non-trivial-apartness-gives-WEM}).
Moreover, we show that injective types have no non-trivial decidable
properties (they are \emph{indecomposable}), unless weak excluded
middle holds (\cref{decomposition-of-injective-type-gives-wem}), which
amounts to a Rice-like theorem for injective types
(cf.~\cite{Rice1953}).

Also, not all types of mathematical structures are injective in
general. For example, the type of (small) inhabited types
is injective if and only if all propositions are
\emph{projective}~\cite{KrausEtAl2017} (a weak choice principle that
fails in some toposes~\cite{FourmanScedrov1982}).
In contrast, the type of pointed types and the type of non-empty types are
injective, highlighting the constructive difference between the double negation
and the propositional truncation.
Similar to the non-injectivity of the type of inhabited types is the
observation that the world's simplest axiom of choice~\cite{FourmanScedrov1982}
holds if the type of two-element types, also known as infinite dimensional
real projective space~\cite{BuchholtzRijke2017} is injective.

The examples of injective types given above are all large types and this is no
coincidence.
In fact, we also have a class of counterexamples by reduction to a
\emph{predicative} taboo with \cref{no-small-injectives} expressing that there
are no nontrivial small injective types unless a weak propositional resizing
principle holds.
This is also a good opportunity to point out that, in the absence of propositional
resizing, the notion of injectivity is universe dependent, as already observed
in~\cite{Escardo2021}, and hence we need to carefully track universes.

One improvement over~\cite{Escardo2021}, recorded as
\cref{injectives-as-algebras}, is that the assumption of propositional resizing
for the characterization of (algebraically) injective types as the algebras of
the lifting monad is not needed.

\subsection{Related work}

Injective sets have already been studied in the context of the constructive and
predicative set theory CZF and the impredicative set theory IZF by
\citeauthor{InjectivesCZF} in the paper~\cite{InjectivesCZF}. In
particular, some statements about injective sets are shown to be equivalent to
(a variant of) weak excluded middle (see e.g.~\cite[Theorem~8]{InjectivesCZF}),
and we add new results in the same vein in this paper.
Moreover, our \cref{no-small-injectives} is comparable to the metatheoretic result \cite[Corollary~10]{InjectivesCZF}, which states that CZF is consistent with the statement that the only injective \emph{sets} (as opposed to classes) are singletons.
Finally, \cref{injective-type-with-non-trivial-apartness-gives-WEM} is somewhat
similar to \mbox{\cite[Lemma~6]{InjectivesCZF}}, although the latter uses a detachable
element, rather than a nontrivial apartness relation.
Our theorem implies that, constructively, no injective type can have an
apartness relation with two points apart. When applied to a type universe, this may
be seen as an internal version of \citeauthor{Kocsis2024}'s metatheoretic result
that Martin-L\"of Type Theory does not define
any non-trivial apartness relation on a universe \mbox{\cite[Corollary~5.7]{Kocsis2024}}. In \emph{loc.\ cit.}\ this fact
is obtained by a parametricity argument.
The papers \cite{EscardoStreicher2016,BooijEtAl2018} discuss the indecomposability of the universe in the absence of weak excluded middle, attributed to Alex Simpson. Here we generalize this to any injective type in \cref{sec:indecomposability}.
The paper~\cite{Kock1991} already observes that algebras of the
lifting monad are flabby, and~\cite{Blechschmidt2018} discusses
connections of flabbiness and injectivity in (the internal language
of) an elementary topos.

\subsection{Agda formalization}

All results in this paper, except Section~\ref{sec:gats:not:formalized}, are
formalized in Agda~\cite{Agda}, building on the TypeTopology
development~\cite{TypeTopology}.
For \cref{no-small-injectives} we make use of
\cite[Theorem~2.13]{deJongEscardo2023} which relies on a construction by
Shulman~\cite[Theorem~5.13]{Shulman2017} that is formalized in the Coq-HoTT
library~\cite[p.~5]{Shulman2017}. We have not ported Shulman's construction to
Agda and instead formulate it as a hypothesis.
The reference~\cite{TypeTopologyPaper} links every numbered environment in this
paper to its implementation. Its HTML rendering is best suited for exploring the
formal development.
We stress that this paper is self-contained and can be read independently from
the formalization.

One possible use of the formalization is to establish precisely which foundational assumptions each definition/lemma/theorem relies on, for readers interested in foundations. In this paper, for simplicity, we make some global assumptions, discussed in Section~\ref{sec:preliminaries}.

\subsection{Outline}

\begin{description}
\item[\cref{sec:preliminaries}]
  We briefly describe the foundational system and notation used in
  this paper, as well as some technical lemmas.
\item[\cref{sec:flabbiness-injectivity}]
  We recall the definitions of (algebraic) flabbiness and injectivity.
  We also present a technical lemma that generalizes the universe
  levels of a lemma in~\cite{Escardo2021}, which is applied to
  characterize the injective types in term of the lifting (also known as the partial map
  classifier) monad, improving on~\cite{Escardo2021} by removing the
  assumption of propositional resizing.
\item[\cref{sec:examples}]
  We list previously known and new examples
  of injective types.  We also give a sufficient condition to derive
  the injectivity of a \(\Sigma\)-type over an injective type. For
  \emph{subtypes} we present a sufficient and necessary criterion.  We
  apply these results to the universe to derive the injectivity of the
  types of various mathematical structures, including
  \(\infty\)-magmas and monoids.
\item[\cref{sec:WEM-DM}]
  We discuss several formulations of weak excluded middle and De
  Morgan's Law and their (non)equivalences, which are relevant for the
  counterexamples in this paper.
\item[\cref{sec:indecomposability}]
  We show that injective types cannot be decomposed in any way unless
  weak excluded middle holds, which amounts to a Rice-like theorem for
  injective types.
\item[\cref{sec:counterexamples}]
  We present various examples of types that cannot be shown to be injective in
  constructive univalent mathematics by reducing their supposed injectivity to
  constructive taboos. We also show that injective types are necessarily large
  unless a weak impredicativity principle holds.

\end{description}

\section{Preliminaries}\label{sec:preliminaries}

The foundations for this paper are as
in~\cite[Section~2.1]{Escardo2021} and can be summarized as: intensional
Martin-L\"of Type Theory with univalent universes (and hence function
extensionality and propositional extensionality), propositional truncations, and
in a few places which are marked explicitly, set quotients.
We mostly adopt the notation and conventions of~\cite{HoTTBook} and write
\(\equiv\) for the definitional (or judgemental) equality, while we reserve
\(=\) for identity types.
We sometimes find it convenient to abbreviate \(\Pi(x : X),Y\,x\) as \(\Pi\,Y\),
and similarly for \(\Sigma\)-types.
We also write \((x : X) \to Y\,x\) for \(\Pi(x : X),Y\,x\) to improve
readability in certain cases.

\subsection{Type universes}
Type universes are rather important in this paper, so we pause to recall the
assumptions of~\cite{Escardo2021}.
We postulate that there is a universe \(\UU_0\), and two operations
\((-)^+\), called successor, and \((-) \sqcup (-)\) satisfying the following
conditions. The operation \({\sqcup}\) is definitionally idempotent, commutative
and associative, and the successor distributes over \({\sqcup}\)
definitionally. Finally, we ask that \(\UU_0 \sqcup \UU \equiv \UU\) and that
\(\UU \sqcup \UU^+ \equiv \UU^+\).
Unlike~\cite{HoTTBook}, we do not require the universes to be
cumulative. However, we stipulate to have an empty type and a unit type in each
universe, which allows one to define lifting functions
\(\UU \to \UU \sqcup \VV\) for any two universes \(\UU\) and \(\VV\), as
in~\cite[Section~2.1]{Escardo2021}.
In this paper, we let $\UU, \VV, \WW, \TT$ range over universes.

\subsection{Fibers, equivalences and embeddings}
We recall that the \emph{fiber} of a map \(f : A \to B\) at
\(b : B\) is defined as
\[
  \fiber_f(b) \coloneqq \Sigma (a : A) , f\,a = b.
\]

Besides equivalences, which may be characterized as those maps whose fibers are
singletons, this paper is concerned with \emph{embeddings}, which may be
characterized as those maps whose fibers are subsingletons. Equivalently,
\({f : A \to B}\) is an embedding precisely when its action on paths
\(
  (a = a') \to (f\,a = f\,a')
\)
is an equivalence for all~\(a,a' : A\).

We also record two consequences of univalence here. First of all, for any two
types \(X\) and \(Y\), we have a map
\(\eqtoid_{X,Y} : (X \simeq Y) \to (X = Y)\) that is an inverse to the canonical
map \(\idtoeq : (X = Y) \to (X \simeq Y)\).
Thanks to this equivalence and the induction principle of identity types, we can
derive \emph{equivalence induction}: given a type family \(P\) indexed over
types \(X\), \(Y\) and an equivalence \(e : {X \simeq Y}\), we get an element of
\(P(X,Y,e)\) at every \(X\), \(Y\) and \(e\), as soon as we can construct an
element of \(P(X,X,\id_X)\).

\subsection{Small types and resizing} \label{small-types-and-resizing}
In this paper we don't assume propositional resizing axioms, but we do discuss mathematical statements that imply forms of resizing that are unprovable of our foundations, and are therefore themselves unprovable.

For a type universe \(\UU\), we say that a type \(X\) in any universe is
\emph{\(\UU\)-small} if it has an equivalent copy in \(\UU\), i.e.\ we have a
type \(Y : \UU\) and an equivalence \(Y \simeq X\).
We extend this notion to maps by saying that a map is \emph{\(\UU\)-small} if all its
fibers are \(\UU\)-small types.

\begin{definition}[\(\weakresizing{\UU}\)]
  We write \emph{\(\weakresizing{\UU}\)} for the assertion that the type
  \[
    \Omeganotnot{\UU} \coloneqq \Sigma (P : \Omega_{\UU}) , (\neg\neg P \to P)
  \] of
  \(\neg\neg\)-stable propositions in \(\UU\), whose native universe is
  \(\UU^+\), is \(\UU\)-small.
\end{definition}

We note that \(\weakresizing{\UU}\) is unprovable, as
observed by Andrew Swan (personal communication) with the following proof. Given a small copy of \(\Omeganotnot{\UU}\), we
can interpret classical second order arithmetic via \(\neg\neg\)-stable
propositions and subsets, but the consistency strength of univalent mathematics
is below that of classical second order arithmetic by
Rathjen's~\cite[Corollary~6.7]{Rathjen2017}.

\subsection{Lemmas involving transport}
We will need the following straightforward lemmas.

\begin{lemma}\label{transport:sigma}
  Suppose we have a type \(X\), a type family \(A\) over \(X\), and a type
  family~\(B\) over \(\Sigma(x : X),A\,x\).
  Given elements \(x,y : X\),
  \(a : A\,x\) and \(b : B(x,a)\), and an identification \(p : x = y\), we have
  \[
    \transportt{\lambda x' \to \Sigma(a : A),B(x',a)}{p}{(a,b)}
      = (\transportt A p a ,\transportd A B p a b)
    \]
    where
    \begin{align*}
      & \transportd A B : (p : x = y) \to A\,x \to B(x,a) \to B(y,\transportt A p a) \\
      & \transportd{A}{B}{(\refl{x})}{a}{b} = b.
    \end{align*}
  \end{lemma}
\begin{proof}
  By path induction.
\end{proof}

\begin{lemma} \label{piproj:sigmainj}
  Let $a$ be an element of a type $X$ and $Y : X \to \UU$ be a type family. Define the \emph{$\Pi$-projection} and the \emph{$\Sigma$-injection} by
  \begin{align*}
    \piproj_a :  \Pi\,Y & \to  Y a & \sigmainj_a :  Y a & \to  \Sigma\,Y \\
              f & \mapsto f\,a & y & \mapsto (a,y).
  \end{align*}
    If $X$ is a proposition with witness $i$, then these maps are equivalences with inverses given by
  \begin{align*}
    \piproj_a^{-1} : Y a & \to \Pi\,Y &     \sigmainj_a^{-1} : \Sigma Y & \to Y a \\
               y & \mapsto \lambda x . \transportt{Y}{(i\,a\,x)}{y} &
              (x , y) & \mapsto \transportt{Y}{(i\,x\,a)}{y}.
  \end{align*}
\end{lemma}
\begin{proof}
  We only give the proof for \(\Pi\), as the one for \(\Sigma\) is
  similar. Note that since propositions are sets, we have \(i\,x\,x = \refl{x}\)
  for any \(x : X\), so that \( {\transportt{Y}{(i\,x\,x)}{y} = y} \) for any
  \(y : Y\,x\).  Moreover, given \(f : \Pi\,Y\) and \(p : x = a\), we can prove
  by path induction on \(p\) that \(\transportt{Y}{(i\,a\,x)}{(f\,a)} = f\,x\).
  Hence, \(\piproj_a^{-1}(\piproj_a f) = f\) for any \(f : \Pi\,Y\), and we
  also have \(\piproj_a(\piproj_a^{-1} y)\) for any \(y : Y\,a\) by the first
  observation.
\end{proof}

\section{Flabbiness and injectivity}\label{sec:flabbiness-injectivity}

When discussing injectivity, it is crucial to pay close attention to universe
levels, as explained in~\cite{Escardo2021}, and we let $\UU, \VV, \WW, \TT$
range over universes.
We recall the following definitions from~\cite{Escardo2021}.

\begin{definition}[Injectivity]
  A type \(D : \WW\) is \emph{algebraically injective} with respect to \(\UU\)
  and \(\VV\) if for every embedding \(j : X \to Y\) with \(X : \UU\) and
  \(Y : \VV\), every map \(f : X \to D\) has a specified extension
  \(f/j : Y \to D\), i.e.\ such that the diagram
  \[
    \begin{tikzcd}
      X \ar[dr,"f"'] \ar[rr,hookrightarrow,"j"] & & Y \ar[dl,"f/j",dashed] \\
      & D
    \end{tikzcd}
  \]
  commutes.
\end{definition}

\noindent
\textbf{\emph{Terminological convention.}}
  Notice that we are asking for a specified extension~\(f/j\) in the above
  definition, i.e.\ expressed with \(\Sigma\) rather than its propositional
  truncation. This is why we say \emph{algebraically} injective. We note that
  \cite{Escardo2021} reserves the terminology \emph{injective} for the notion phrased with
  propositional truncation instead.
  However, in this paper we will consider only algebraic injectivity, so
  henceforth we shall simply speak of injectivity for short when we mean algebraic
  injectivity.
We will also say that a type~\(D\) is \emph{\(\UU,\VV\)-injective} if \(D\) is
injective with respect to \(\UU\)~and~\(\VV\). If \(\UU \equiv \VV \equiv \TT\),
then we simply say that \(D\) is \emph{\(\TT\)-injective}.

Specializing the extension problem to propositions \(P\) with their embedding
into the unit type \(\One\), we arrive at the notion of flabbiness, which we
define now and for which the same terminological convention applies.

\begin{definition}[Flabbiness]
  A type \(D\) is \emph{\(\UU\)-flabby} if for every proposition \(P : \UU\),
  every partial element \(\varphi : P \to D\) has a specified extension to a (total)
  element \(d : D\) such that \(\varphi\,p = d\) whenever we have \(p : P\).
\end{definition}
This is summarized by the diagram
  \[
    \begin{tikzcd}
      P \ar[dr,"\varphi"'] \ar[rr,hookrightarrow,""] & & \One \ar[dl,"d",dashed] \\
      & D,
    \end{tikzcd}
  \]
identifying a (total) element of \(D\) with a function \(\One \to D\).
We will often use the following lemmas tacitly.
\begin{lemma}[{\cite[Lemmas~19~and~20]{Escardo2021}}]\label{flabby-iff-injective}
  If a type \(D\), in any universe \(\WW\), is \(\UU,\VV\)-injective, then it is
  \(\UU\)-flabby.
  In the other direction, if \(D\) is \(\UU\sqcup\VV\)-flabby, then it is
  \(\UU,\VV\)-injective.
\end{lemma}

\begin{lemma}[{\cite[Lemma~12]{Escardo2021}}]\label{closure-under-retracts}
  If a type \(D\), in any universe \(\WW\), is
  \(\UU,\VV\)-injective, then so is any retract \(D' : \WW'\) of
  \(D\) in any universe \(\WW'\).
\end{lemma}

\subsection{Resizing lemmas}\label{sec:injectivity-resizing}

The following lemma modifies one direction of~\cref{flabby-iff-injective} by
generalizing the flabbiness universe from $\UU \sqcup \VV$ to any universe
$\TT$, at the expense of requiring the embedding $j$ to be a $\TT$-small map.
\begin{lemma}
  If $D : \WW$ is a $\TT$-flabby type, $X: \UU$ and $Y: \VV$ are types, and
  $j: X \to Y$ is a $\TT$-small embedding, then any map $f : X \to D$ extends to
  a map $f' : Y \to D$ along $j$.
\end{lemma}
\begin{proof}
  Since the map \(j\) is assumed to be \(\TT\)-small, we have a type family
  \({R : Y \to \TT}\) and an equivalence \(\rho_y : R\,y \simeq \fiber_j(y)\) for
  all \(y : Y\).
  Since \(j\) is an embedding, its fibers are propositions and hence \(R\) is a
  family of propositions as well.
  For \(y : Y\), consider the partial element \(\tilde{f}_y : R\,y \to D\) given by
  \(\tilde{f}_y\,r = f (\fst (\rho_y \, r))\).
  Since \(D\) is \(\TT\)-flabby, we obtain, for every \(y : Y\), an element
  \(d_y\) extending \(\tilde{f}_y\).
  We claim that the assignment \(\overline{f} : Y \to D\) given by
  \(y \mapsto d_y\) extends \(f\).
  Indeed, given \(x : X\), we have
  \begin{align*}
    f\,x &\equiv (f \circ \fst) (x, \refl{(j\,x)}) \\
         &= \pa*{f \circ {\fst} \circ {\rho_{j\,x} \circ \rho_{j\,x}^{-1}}} (x , \refl{(j\,x)}) \\
         &\equiv \pa[\big]{\tilde{f}_{j\,x} \circ \rho_{j\,x}^{-1}} (x,\refl{(j\,x)})
           &\text{(by definition of \(\tilde{f}\))} \\
         &= d_{j\,x}
           &\text{(since \(d_{j\,x}\) extends \(\tilde{f}_{j\,x}\))} \\
         &\equiv \overline{f}(j\,x),
  \end{align*}
  as desired.
\end{proof}

\begin{lemma}\label{ainjectivity-over-small-maps}
  If $D : \WW$ is injective with respect to universes $\TT_0 \sqcup \TT_1$ and
  $\TT_2$, $X: \UU$ and $Y: \VV$ are types, and $j: X \to Y$ is a $\TT_0$-small
  embedding, then any map $f : X \to D$ extends to a map $f' : Y \to D$ along
  $j$.
\end{lemma}
\begin{proof}
  By the previous lemma and \cite[Lemma~19]{Escardo2021}.
\end{proof}

The following generalizes the universe levels of \cite[Lemma~14]{Escardo2021}.
\begin{lemma}\label{embedding-retract}
    If \(D : \UU\) is injective with respect to universes \(\TT_0 \sqcup \TT_1\) and \(\TT_2\), then \(D\) is a retract of \(Y\) for any type \(Y : \VV\) with a \(\TT_0\)-small embedding \(D \hookrightarrow Y\), in which case the embedding is the section.
\end{lemma}
\begin{proof}
    Given a \(\TT_0\)-small embedding \(j : D \hookrightarrow Y\), we apply \cref{ainjectivity-over-small-maps} with \(f \coloneqq \id_D : D \to D\).
\end{proof}

\subsection{Algebras of the lifting monad}\label{sec:characterization-injectives}

We first briefly recall the lifting monad~\cite[Section 2]{escardo:knapp:2017}. Given a type $X : \UU$ and any universe $\TT$, we define
\[
  \lifting_{\TT} X \coloneqq \Sigma (P : \Omega_{\TT}) , (P \to X).
\]
This lives in the universe $\TT^+ \sqcup \UU$. The unit of the monad is given by
\begin{align*}
  \eta_X : X & \to \lifting_{\TT} X \\
  x & \mapsto (\One , (\star \mapsto x)).
\end{align*}
For further details of the monad structure and laws, see \emph{loc.\ cit.}, but they are not needed for our purposes.

Theorem 51 of~\cite{Escardo2021} assumes propositional resizing to show that the algebraically injective types are precisely the retracts of algebras of the lifting monad. We now remove the resizing assumption by an application of \cref{embedding-retract} above.
\begin{lemma}\label{eta-is-small-map}
    For any \(X : \UU\), the unit of the lifting monad \(\eta : X \to \lifting_{\TT} X\) is a \(\TT\)-small map.
\end{lemma}
\begin{proof}
    We claim that the fiber of \(\eta\) at \((P : \TT,i : \isprop P,{\varphi : P \to X}) : \lifting_{\TT} X\) is equivalent to the proposition \(P : \TT\). By \cite[Lemma~45]{Escardo2021}, the unit is an embedding so its fibers are propositions, hence it suffices to construct maps in both directions.
    Given \(p : P\), we have \((P,i,\varphi) = \eta(\varphi\,p)\), and if \((P,i,\varphi) = \eta\,x\) for \(x : X\), then \(P\) holds as the domain of definition of \(\eta\,x\) is the unit type.
\end{proof}

As a corollary, we can improve the universe levels of \cite[Lemma~50]{Escardo2021}.
\begin{lemma}\label{injective-is-retract-of-free-lifting-algebra}
    If \(D : \UU\) is injective with respect to universes \(\VV \sqcup \TT\) and \(\WW\), then
    \(D\) is a retract of \(\lifting_{\TT} D\).
\end{lemma}
\begin{proof}
    Apply \cref{embedding-retract} to the embedding \(\eta : D \to \lifting_{\TT} D\) which is \(\TT\)-small by virtue of \cref{eta-is-small-map}.
\end{proof}
In light of \cref{sec:counterexamples} below, which shows that there are no small injective types unless a certain propositional resizing principle holds, a particular case of interest is obtained by taking \(\UU = \TT^+\) in the following consequence of this lemma.
\begin{theorem}
    A type \(D : \UU\) is $\TT$-injective if and only if we have \(X : \UU\) such that \(D\) is a retract of \(\lifting_{\TT} X\), i.e.\ \(D\) is a retract of a free \(\lifting_{\TT}\)-algebra.
\end{theorem}
\begin{proof}
    The left to right direction follows directly from \cref{injective-is-retract-of-free-lifting-algebra}. The other implication holds because injective types are closed under retracts and because lifting monad algebras are injective (\cite[Lemma~48]{Escardo2021}).
\end{proof}

Finally, we obtain the following result, which removes the assumption of propositional resizing from \cite[Theorem~51]{Escardo2021}. Again, a particular case of interest is obtained by taking \(\UU = \TT^+\) in the theorem.
\begin{theorem}\label{injectives-as-algebras}
    A type \(D : \UU\) is $\TT$-injective if and only if \(D\) is a retract of an  \(\lifting_{\TT}\)-algebra.
\end{theorem}
\begin{proof}
    In the left to right direction, we assume that \(D : \UU\) is injective with respect to \(\TT\) and we consider the free algebra on \(\lifting_{\TT} D\) of which \(D\) is a retract by \cref{injective-is-retract-of-free-lifting-algebra}.
    The other implication holds because injective types are closed under retracts and because lifting algebras are injective (\cite[Lemma~48]{Escardo2021}).
\end{proof}

\section{Examples}\label{sec:examples}
Examples of injective types include the following:
\begin{enumerate}
\item The universe $\UU$ (by \cite[Lemma~3]{Escardo2021}). In fact, the universe
  is injective in at least two ways. Given an embedding
  \(j : X \hookrightarrow Y\) and a type family \(f : X \to \UU\), we may
  construct an extension \(f/j : Y \to \UU\) by setting
  \((f/j) (y) \coloneqq \Sigma((x,-) : \fiber_j(y)),f\,x\). Alternatively, we
  can use \(\Pi\) instead of \(\Sigma\) in the definition of the extension.
\item The type \(\Omega_{\UU}\) of propositions in $\UU$ (by
  \cite[Theorem~24]{Escardo2021}), and more generally any reflective subuniverse
  (see \cite[Section~7]{Escardo2021}, or alternatively
  \cref{reflective-subuniverse-injective} below).
  In all examples we know, univalence is needed in order to establish
  injectivity. In this example, a landmark particular case of
  univalence is sufficient, namely propositional extensionality.

\item The types of iterative multisets and sets in $\UU$ (by \cref{sec:injective-iterative-sets} below).
\item The types of pointed types, monoids, groups, \(\infty\)-magmas, and pointed
  \(\infty\)-magmas in $\UU$ (by the results in
  \cref{sec:injectivity-mathematical-structures} below).
\item The underlying sets of sup-lattices and pointed directed complete posets
  (by \cref{sec:injective-carriers-of-posets} below).
\item The type of ordinals in $\UU$~\cite[Section~10.3]{HoTTBook}.
  The type of ordinals can be equipped with injectivity data in multiple
  ways. One way is to use the fact that this type is the carrier of a
  sup-lattice~(\cite[Theorem~5.8]{deJongEscardo2023}), at least in the presence
  of small set quotients.
  Another injectivity structure is used in~\cite{Escardo2019} and will be
  detailed in a forthcoming paper.
\end{enumerate}

We also mention another example that we will revisit in
\cref{sec:injectivity-of-Sigma-types}. This example should be contrasted to
\cref{type-of-inhabited-types-injective-iff-projective-props} which tells us
that the type of \emph{inhabited} types is not in general injective.

\begin{proposition}\label{type-of-nonempty-types-is-injective}
  The type \(\NE \coloneqq \Sigma (X : \UU) , \lnot\lnot X\) of nonempty types
  in a universe \(\UU\) is a retract of \(\UU\) and hence injective.
\end{proposition}
\begin{proof}
  We define \(s : \NE \to \UU\) as the first projection and \(r : \UU \to \NE\)
  by taking a type \(X\) to the type \(\lnot\lnot X \to X\) which is nonempty
  because \(\lnot\lnot(\lnot\lnot X \to X)\) holds (constructively).
  We need to show that \(s\) is a section of \(r\). By univalence and
  the fact that nonemptiness is property, it suffices to prove that
  \(\lnot\lnot X \to X\) and~\(X\) are equivalent when \(X\) is
  nonempty. But in this case \(\lnot\lnot X \simeq \One\) holds and we
  always have \({(\One \to X)} \simeq X\), so we are done.
\end{proof}

Alternatively, we can use the techniques of the upcoming
\cref{sec:injectivity-mathematical-structures} to prove that the type of
nonempty types injective, because double negation is closed under
proposition-indexed products.
Indeed, we have the following, which may be somewhat surprising, as double negation is in
general not closed under products indexed by sets: in the presence of excluded middle it implies the axiom of choice~\cite{TypeTopologyChoice}.
\begin{lemma} \label{propositional-double-negation-shift}
  The implication
\[
  (\Pi (p : P), \neg\neg X_p) \to \neg\neg \Pi (p : P), X_p
\]
holds constructively for any proposition $P$ and any family $X$ of types indexed by~$P$.
\end{lemma}
\begin{proof}
  Assume we have
\(\varphi : \Pi (p : P), \neg\neg X_p\) and \(\nu : \neg \Pi (p : P), X_p\), and
we look to derive a contradiction.
Since \(P\) is a proposition, it follows from \(\nu\) that we have \(\neg X_p\)
for all \(p : P\).
Together with \(\varphi\) this implies the negation of \(P\), which makes \(\Pi (p : P), X_p\) vacuously true, contradicting \(\nu\).
\end{proof}

\subsection{Carriers of sup-lattices and pointed dcpos} \label{sec:injective-carriers-of-posets}

By a \emph{pointed \(\VV\)-dcpo} we mean a poset with a least element \(\bot\)
and suprema (least upper bounds) for directed families indexed by types in
\(\VV\), see also~\cite{deJongEscardo2021a,deJongThesis}.

\begin{proposition}\label{pointed-dcpos-are-injective-as-sets}
  The carrier of any pointed \(\VV\)-dcpo is a \(\VV\)-injective type.
\end{proposition}

Note also that this implies that the carrier of any sup-lattice is injective.

\begin{proof}
  Let \(D\) be the carrier of a pointed \(\VV\)-dcpo and suppose that a proposition
  \(P : \VV\) and a map \(f : P \to D\) are given. The family
  \({\One + P} \to D\) mapping \(\inl \star\) to \(\bot\) and \(\inr p\) to
  \(f\,p\) is directed and hence has a supremum \(d\) in \(D\). Being the
  supremum, it is easy to see that \(d = f\,p\) whenever we have \(p : P\).
\end{proof}

\begin{example}\label{extended-partial-Dedekind-reals}
  Later we will see that the type \(\R\) of Dedekind reals is not in general
  injective~(\cref{R-injective-taboos}).
  However, the \emph{extended partial} Dedekind reals are: these generalize
  Dedekind reals by dropping the inhabitedness and locatedness condition. Recall
  that a Dedekind cut \(x\) is located if \(p < x\) or \(x < q\) for every two
  rationals \(p < q\). The inhabitedness (or boundedness) condition says that there
  exist rationals \(p\) and \(q\) with \(p < x\) and \(x < q\).
  Thus, an extended partial Dedekind real \(x\) is given by a pair of subsets
  \(L_x,U_x \subseteq \mathbb Q\) satisfying the following conditions, where we
  (suggestively) write \(p < x\) for \(p \in L_x\) and \(x < q\) for
  \(q \in U_x\):
  \begin{enumerate}[(i)]
  \item transitivity: \(p < x\) and \(x < q\) together imply \(p < q\);
  \item roundness:
    \(p < x \iff \exists (r : \mathbb Q),\, (p < r) \times (r < x)\) and
    similarly, \(x < q \iff \exists (s : \mathbb Q),\, (x < s) \times (s < q)\).
  \end{enumerate}
  The type of extended partial Dedekind reals may be equipped with the structure
  of a directed complete poset, by defining \(x \sqsubseteq y\) as
  \((L_x \subseteq L_y) \times (U_x \subseteq U_y)\).
  Indeed, this order is easily checked to be antisymmetric, reflexive and
  transitive, and the directed join of a family \((x_i)_{i : I}\) of extended
  partial Dedekind reals is giving by \(x_\infty\) where
  \(L_{x_\infty} \coloneqq \bigcup_{i : I}L_{x_i}\) and
  \(U_{x_\infty} \coloneqq \bigcup_{i : I}U_{x_i}\).
  The transitivity condition for \(x_\infty\) is verified using that the family
  \((x_i)_{i : I}\) is directed.
  Finally, this poset has a least element given by a pair of empty sets.
  Thus, the injectivity of the extended partial Dedekind reals follows from
  \cref{pointed-dcpos-are-injective-as-sets}.
\end{example}

Notice that the extended partial reals are the points of the formal
space constructed by Sara Negri~\cite[Section 3]{Negri2002}.

\subsection{The types of iterative sets and multisets}\label{sec:injective-iterative-sets}

We briefly recall the types of iterative multisets and iterative sets, as
constructed by \citeauthor{Gylterud2018}~\cite{Gylterud2018,Gylterud2019}.
The type \(\bM\) of \emph{iterative multisets} is inductively defined by
a single constructor
\[
  \ssup : \Pi (X : \UU),(X \to \bM) \to \bM.
\]
For a multiset \(\ssup(X,\varphi)\), we refer to \(X\) as its \emph{root} and to
the map \(\varphi : X \to \bM\) as its \emph{forest}.
A multiset \(\ssup(X,\varphi)\) is said to be an \emph{iterative set} if
\(\varphi\) is an embedding and \(\varphi\,x\) is an iterative set for each
\(x : X\).
The type \(\bV\) of \emph{iterative sets} is the subtype of \(\bM\) restricted
to iterative sets.

\begin{theorem}
  The type of iterative multisets in $\MCU$ is $\MCU$-flabby.
\end{theorem}
\begin{proof}
  Given a partial element $A : P \to \bM$, we extend it to the multiset $M$ with
  root $\Sigma (p : P) , \iroot (A \, p)$ and forest
  $(p , a) \mapsto \forest (A \, p) \, a$. Because $P$ is a proposition, the
  type $\Sigma (p : P) , \iroot (A \, p)$ is equivalent to $\iroot (A \, p_0)$
  for any $p_0 : P$ by the map
  \[
    (p , a) \mapsto \transportt{A}{(i \, p \, p_0)}{a}
  \]
  where $i$ witnesses the assumption that $P$ is a proposition.  Then we have
  \begin{align*}
    M & = \ssup (\Sigma (p : P) , \iroot (A \, p)), ((p,a) \mapsto \forest(A \, p) a) \\
      & = \ssup (\iroot(A \, p_0),\forest(A \, p_0)) \\
      & = A \, p_0
  \end{align*}
  where we use univalence, the characterization of equality of multisets, and
  the propositional $\eta$-equality for multisets.
\end{proof}

\begin{corollary}\label{iterative-multisets-injective}
  The type of iterative multisets in \(\MCU\) is \(\UU\)-injective.
\end{corollary}

\begin{theorem}\label{iterative-sets-injective}
  In the presence of set quotients, the type of iterative sets in \(\MCU\) is
  \(\UU\)-injective.
\end{theorem}
\begin{proof}
  Since \(\bM\) is injective, it suffices to exhibit \(\bV\) as retract of
  \(\bM\). Let us write \(s : \bV \hookrightarrow \bM\) for the canonical
  inclusion. We define \(r : \bM \to \bV\) recursively via
  \[
    r(\ssup(S,\varphi)) \coloneqq \bigcup_{x : X}r(\varphi\,x),
  \]
  which uses that \(\bV\) has unions.
  We comment on their construction and properties before showing that \(r\) is a
  retraction of \(s\).

  Given a family of iterative sets \(\mathcal A : I \to \bV\) with \(I : \UU\) a
  small type, we construct its union following~\cite[Lemma~5.6]{Gylterud2018}.
  We would like to consider the iterative set with the image \(\im \mathcal A\)
  of \(\mathcal A\) as its root and the canonical inclusion
  \(\im\mathcal A \hookrightarrow \bV\) as its forest.
  Since \(\bV : \UU^+\) is a large type, \(\im\mathcal A\) is also large, so
  that we cannot do this for size reasons.
  However, in the presence of set quotients, we can prove the set replacement
  principle (in fact the two are equivalent~\cite{deJongEscardo2023}) which says
  that the image of a map from a small type to a locally small set is small.
  Here, we recall that a type is locally
  small~(\cite[Definition~4.1]{Rijke2017}) if all of its identity types are
  small, and that \(\bV\) is locally small~\cite{Gylterud2018} (as equality of
  iterative sets is characterized by the membership relation).
  Hence, we get \(S : \UU\) with \(\psi : S \hookrightarrow \bV\) making the diagram
  \[
    \begin{tikzcd}
      {\im\mathcal A} \ar[d,hookrightarrow] \ar[r,"\simeq"] & S  \\
      \bV \ar[ur,hookleftarrow]
    \end{tikzcd}
  \]
  commute.
  In particular we can consider
  \(\bigcup\mathcal A \coloneqq \ssup(S,\psi) : \bV\).

  Now, following~\cite{Gylterud2018}, we can define an extensional set
  membership relation \(\in\) on \(\bV\) and the union is characterized as
  expected:
  \[
    B \in \bigcup\mathcal A \iff \exists(i : I),\,B = \mathcal A_i.
  \]
  In particular it follows that for iterative sets \(A\) and \(\ssup(X,\varphi)\) we have
  \[
    A \in \ssup(X,\varphi) \iff A \in \bigcup_{x : X}\varphi\,x,
  \]
  so that any iterative set is equal to the union of its forest.

  Finally, we prove that \(r\) is a retraction of \(s\) by induction: for an
  arbitrary iterative set \(\ssup(X,\varphi)\) we have
  \begin{align*}
    r(s(\ssup(X,\varphi))) &= r(\ssup(X,{s \circ \varphi})) \\
    &\equiv \bigcup_{x : X}r(s(\varphi\,x)) \\
    &= \bigcup_{x : X} \varphi\,x &\text{(by induction hypothesis)} \\
    &= \ssup(X,\varphi),
  \end{align*}
  as any iterative set is the union of its forest.
\end{proof}

\subsection{Types of mathematical structures}\label{sec:injectivity-mathematical-structures}

We now show that several types of mathematical structures, including
pointed types, monoids, $\infty$-magmas and groups, are flabby, or,
equivalently, injective.
We work with an arbitrary $S : \UU \to \VV$ and want to show that the
type $\Sigma (X : \UU) , S\,X$, abbreviated as $\Sigma S$, is
flabby. E.g.\ for $\infty$-magmas in a universe~$\UU$, we will take
$\VV$ the same as $\UU$ and $S\,X = X \to X \to X$, specifying that we
have a binary operation on the type $X$, subject to no axioms, and the
type of $\infty$-magmas will be $\Sigma S$. Similarly, the type of
groups will again be of the form $\Sigma S$ for a different choice
of~$S$.

In outline, to show that the type $\Sigma S$ is flabby, it is enough
to show that \(S\) is closed under proposition-indexed products. Let
$P$ be a proposition and $f : P \to \Sigma S$ be a partial element of
$\Sigma S$. In order to attempt to extend $f$ to a total element of
$\Sigma S$, first observe that $f$ is of the form
\begin{align*}
  f\,p = (A\,p , g\,p)
\end{align*}
with
\begin{align*}
  A  : P \to \UU, \qquad
  g  : \Pi (p : P) , S(A\,p).
\end{align*}
So we need to construct $(X , s) : \Sigma S$ such that
\begin{align*}
  (X , s) = (A\,p , g\,p)
\end{align*}
and hence
\begin{align*}
  X = A\,p
\end{align*}
for all \(p : P\). Because $P$ is a proposition, we have, for $p : P$, an
equivalence
\begin{align*}
\piproj_p : \Pi\,A \simeq A\,p,
\end{align*}
given by projection out of the product, which, in turn, by univalence,
induces an identification
\begin{align*}
  \phi_p : \Pi\,A =  A\,p,
\end{align*}
as shown in \cref{piproj:sigmainj}. Hence we can take
\begin{align*}
  X \coloneqq \Pi\,A.
\end{align*}
To construct the required $s : S\,X$, we need an assumption on
$S$. As discussed above, our assumption will be that $S$ is closed under
proposition-indexed products. Roughly, this means that from an element
of the type $\Pi (p : P) , S(A\,p)$ we can get an element of the type
$S(\Pi\,A)$.  To make this precise, we consider the
transport-along-equivalences map $\treq_{X,Y}$, for $X,Y : \UU$,
defined by
\begin{align*}
  \treq_{X,Y} :  X \simeq Y & \to S\,X \to S\,Y \\
  h & \mapsto \transport^S (\eqtoidt{X}{Y}{h}).
\end{align*}
We will omit the subscripts when they are clear from the context.
Because transports are always equivalences, the map
$\treq h : S X \to S Y$ is an equivalence. We define
\begin{align*}
  \rho_{P,A} : S(\Pi\,A) & \to \Pi (p : P) , S(A\,p) \\
  s & \mapsto \lambda p. \treq \,\piproj_p \, s,
\end{align*}
where we omit the subscripts when they can be inferred from the context.
We refer to $\rho$ as the \emph{canonical map associated to $S$}.
\begin{definition}
  We say that $S$ is \emph{closed under proposition-indexed products}
  if the map $\rho_{P,A}$ is an equivalence for all propositions $P$
  and families $A : P \to \UU$.
\end{definition}
\begin{lemma} \label{mathematical-structures-flabbiness-criterion}
  If $S$ is closed under proposition-indexed products, then $\Sigma S$ is $\UU$-flabby.
\end{lemma}
\begin{proof}
  Continuing from the above, it remains to construct $s : S X$ such that
\[ (X , s) = (A\,p , g\,p) \] for every $p : P$, where we have taken $X \coloneqq \Pi\,A$. We take
\begin{align*}
  s \coloneqq \rho_{P,A}^{-1} (g) : S (\Pi\,A).
\end{align*}
Then, omitting the subscripts $P,A$ for the sake of brevity, for every $p : P$, we have that
\begin{align*}
    \transport^S \, (\phi \, p) \, s &\equiv \rho \, s\,p \\
  &\equiv  \rho \, (\rho^{-1} g) \, p \\
  & =  g \, p.
\end{align*}
Therefore, using the characterization of equality in $\Sigma$ types,
\begin{align*}
  (X , s) &\equiv (\Pi\,A , \rho^{-1} g) \\
  & =  (A \, p , g \, p) \\
  &\equiv  f \, p.
\end{align*}
Thus the pair $(X , s)$ does indeed extend $f$ to a total element, as required.
\end{proof}
Now, in order to simplify the application of the above flabbiness criterion, we assume we
are given arbitrary maps
\begin{align*}
  T_{X,Y} & : X \simeq Y \to (S X \to S Y) \\
  r_{X} & : T_{X,X} \id \sim \id.
\end{align*}
The point is that any such $T$ can be equivalently expressed as a
transport, as observed below, but it is more convenient to work with an appropriately
chosen~$T$ equipped with the data~$r$ in practice. Examples of choices
of $T$ are given in~\cref{examples:of:flabby:mathematical:structures}
below.
\begin{lemma} \label{treq:lemma}
  For any $h : X \simeq Y$, we have that $T\,h \sim \treq h$, and hence
  \begin{align*}
    \rho \, s \, p = T \, \piproj_p \, s
  \end{align*}
  for any $s : S(\Pi\,A)$.
\end{lemma}
\begin{proof}
  By equivalence induction, it suffices to prove this for
  $h \coloneqq \id$, which amounts to the assumption $r$. For the
  consequence, we consider $h \coloneqq \piproj_p$ with
  $X \coloneqq \Pi\,A$ and $Y \coloneqq A\,p$.
\end{proof}
We apply the following to combine mathematical structures (e.g.\ to combine
pointed structure and $\infty$-magma structure to get pointed $\infty$-magma
structure).
\begin{lemma} \label{product-of-structures} If $S_1 : \UU \to \VV_1$
  and $S_2 : \UU \to \VV_2$ are closed under proposition-indexed
  products, then so is
  \begin{align*}
    S : \UU & \to \VV_1 \sqcup \VV_2 \\
    X & \mapsto S_1 X \times S_2 X.
  \end{align*}
\end{lemma}
\begin{proof}
  Suppose we are given a proposition \(P\) and a family $A : P \to \UU$.
  Let \(\rho_1\) and~\(\rho_2\) be the equivalences that witness the closure under proposition-indexed products for $S_1$ and $S_2$. Then an inverse of $\rho : S(\Pi\,A) \to \Pi (p : P) , S(A\,p)$ is given by
  \begin{align*}
    \rho^{-1} \,\alpha = (\rho_1^{-1} (\fst \comp \alpha),\rho_2^{-1} (\snd \comp \alpha)),
  \end{align*}
  as is easy to verify.
\end{proof}
We apply the following to get sub-structures by considering axioms.
For instance, if $S$ is pointed $\infty$-magma structure, then, to get
monoids, we choose $\axioms$ that specifies that the underlying type is a
set and that the operation is associative with the distinguished point as a
left and right neutral element, so that $\Sigma \SA$ will be type of monoids.
We consider axioms to be given by properties which is why $\axioms$ takes values
in $\Omega$ in the lemma below.
\begin{lemma}\label{axioms-closed-under-products}
  Suppose that $S$ is closed under proposition-indexed products
  and that
  \begin{align*}
    \axioms : \Pi (X : \UU), \left(S\,X \to \Omega_{\WW}\right)
  \end{align*}
  is given, and define
  \begin{align*}
    \SA : \UU & \to \VV \sqcup \WW \\
    X & \mapsto \Sigma (s : S \,X) , \axioms \, X \, s.
  \end{align*}
  If we have a (necessarily unique) map
  \begin{align*}
    \axiomspi & : (\alpha : (p : P) \to S (A \, p))  \\
              & \to ((p : P) \to \axioms\, (A \, p) \, (\alpha \, p)) \\
              & \to \axioms \, (\Pi\,A) \, (\rho^{-1}\,\alpha),
  \end{align*}
  then $\SA$ is closed under proposition-indexed products.
\end{lemma}
\begin{proof}
 It is easy to verify, using~\cref{transport:sigma}, that
  \begin{align*}
    \rhoa^{-1} : (\Pi (p : P), \SA(A\,p)) & \to \SA(\Pi\,A) \\
    \alpha & \mapsto (\rho^{-1} (\fst \comp \alpha) , \axiomspi \, (\fst \comp \alpha) \, (\snd \comp \alpha))
  \end{align*}
  is an inverse of the canonical map $\rhoa$ associated to $\SA$, where
  \begin{align*}
    \rho : S(\Pi\,A) \to \Pi (p : P) , S(A\,p),
  \end{align*}
    is the canonical map associated to $S$.
\end{proof}

We now apply the above lemmas, with different choices of $S$, $T$ and
$r$, to get the following examples.
When dealing with structure, we conveniently make use of an appropriate chosen
$T$, thanks to \cref{treq:lemma}, while for the axioms we can work directly
with $\rho$ and \cref{axioms-closed-under-products}.
\begin{examples} \label{examples:of:flabby:mathematical:structures}
  The following are flabby and hence injective.
  \begin{enumerate}
  \item \label{pointed} The type of pointed types in a universe $\UU$.
  \item The type of $\infty$-magmas in a universe $\UU$.
  \item The type of pointed $\infty$-magmas in a universe $\UU$.
  \item The type of monoids in a universe $\UU$.
  \item The type of groups in a universe $\UU$.
  \item\label{exa:graphs} The type of graphs in a universe $\UU$.
  \item The type of posets in a universe $\UU$.
  \item The type of families in a universe $\UU$.
  \item The type of functions in a universe $\UU$.
  \end{enumerate}
\end{examples}
\begin{proof}
\emph{The type of pointed types.} We take $S\,X \coloneqq X$ and choose
  \begin{align*}
    T & : X \simeq Y \to (X \to Y)
  \end{align*}
  that maps an equivalence to its underlying function. Because it
  maps the identity equivalence to the identity function, we have the required
  \begin{align*}
    r & : T \id \sim {\id}.
  \end{align*}
  By \cref{treq:lemma}, we see that $\rho_{P,A} \sim \id_{S(\Pi\,A)}$ for any
  $P : \Omega_\UU$ and $A : P \to \UU$, and hence is an equivalence, and so the
  result follows from \cref{mathematical-structures-flabbiness-criterion}.

  \emph{The type of $\infty$-magmas.} We take $S\,X = (X \to X \to X)$ and we choose
  \begin{align*}
    T & : X \simeq Y \to (S\,X \to S\,Y)
  \end{align*}
  that maps an equivalence $f$ and a binary operation $\ast$ to the operation $\star$
  defined by $y \star y' \coloneqq f({f^{-1}\, y} \mathbin{\ast} {f^{-1} \, y'})$. It is direct that this maps the identity equivalence to the identity function on binary operators (using function extensionality twice).
  To get the flabbiness of $\Sigma S$,
  \cref{mathematical-structures-flabbiness-criterion} requires an inverse
  $\rho_{P,A}$ for $P : \Omega_{\UU}$ and $A : P \to \UU$.
  We define the required inverse by
  \begin{align*}
    \rho^{-1} : (\Pi (p : P) , S(A\,p)) & \to S (\Pi\,A) \\
    g & \mapsto \lambda \alpha \, \beta \, p \, . \, g\, p\, (\piproj_p \alpha) \, (\piproj_p \beta),
  \end{align*}
  and it is straightforward to check that this is an inverse of $\rho$ using
  \cref{piproj:sigmainj} and the fact that we have
  \(\rho \, s \, p = T \, \piproj_p \, s\) for any $s : S(\Pi\,A)$ and for the
  above~$T$ by \cref{treq:lemma}.

  \emph{The type of pointed $\infty$-magmas.} This follows from the previous two examples combined with the aid of \cref{product-of-structures}.

  \emph{The type of monoids.} A monoid is a pointed $\infty$-magma whose underlying type is a set, whose distinguished point serves as a unit (which is property as the underlying type is a set) and whose multiplication is associative (which is also property for the same reason), and hence this example follows from the previous three examples with the aid of \cref{axioms-closed-under-products}, after we check that the operation $y \star y' \coloneqq f(f^{-1} y \ast f^{-1} y')$ satisfies the monoid axioms, which is straightforward.

  \emph{The type of groups.} A group may be considered as a monoid satisfying a further axiom (the (necessarily unique) existence of a two-sided inverse), and hence we may apply \cref{axioms-closed-under-products} again to the example of monoids, after checking that $y \star y' \coloneqq f(f^{-1} y \ast f^{-1} y')$ satisfies this further axiom.

  \emph{The type of graphs.} We consider $S : \UU \to \UU^+$ defined by \({S\,X \coloneqq (X \to X \to \UU)}\)
  and choose
  \begin{align*}
    T : X \simeq Y & \to (S\,X \to S\,Y) \\
    f & \mapsto (\lambda R \,y\, y'.\, R\,(f^{-1}\,y)\,(f^{-1}\,y')).
  \end{align*}
  To get the injectivity of $\Sigma S$,
  \cref{mathematical-structures-flabbiness-criterion} requires an inverse
  $\rho_{P,A}$ for $P : \Omega_{\UU}$ and $A : P \to \UU$.
  We define the required inverse by
  \begin{align*}
    \rho^{-1} : (\Pi (p : P) , S(A\,p)) & \to S (\Pi\,A) \\
    g & \mapsto \lambda \alpha \, \beta \, . \, \Pi (p : P) , g\,p\, (\alpha\,p)\,(\beta\,p),
  \end{align*}
  and it is straightforward to check that this is an inverse of $\rho$ using
  \cref{piproj:sigmainj} and the fact that we have
  \(\rho \, s \, p = T \, \piproj_p \, s\) for any $s : S(\Pi\,A)$ by~\cref{treq:lemma}.

  \emph{The type of posets.} Using the injectivity of the type of graphs, by
  \cref{axioms-closed-under-products}, we only need to check that, given a
  partial order \(\sqsubseteq\) on \(A\,p\), the relation \(\leq\) on \(\Pi\,A\)
  given by \(g \leq h \coloneqq \Pi(p : P) , g\,p \sqsubseteq h\,p\) is again a
  partial order, which is straightforward.

  \emph{The type of families.} This is similar to, and easier than, the type of graphs, with $S \, X \coloneqq (X \to \UU)$.

  \emph{The type of functions.} The type of families \(X \to \UU\) with
  \(X : \UU\) is equivalent to the type of functions \(Y \to X\) with
  \(X : \UU\) and \(Y : \UU\).
\end{proof}

\subsection{\texorpdfstring{$\Sigma$}{Sigma}-types}
\label{sec:injectivity-of-Sigma-types} \label{injectivity-of-sigma-types}

We now generalize the results of the previous subsection in three ways:
\begin{enumerate}
\item Instead of working with $S : \UU \to \VV$, we work with a family
  $A : X \to \VV$, generalizing the type universe $\UU$ to an
  arbitrary type \(X\) in any universe, and we give
  a sufficient condition on $X$ and $A$ for the flabbiness of the type
  \(\Sigma (x : X) , A\,x\) from the flabbiness of the type~$X$.

\item We allow more general flabbiness structures on $X$.

\item We relax the requirement of having an inverse in the definition
  of closure under flabby structures to having a section, so that the
  condition that was property now becomes data.
\end{enumerate}

We assume a witness $\phi$ of the $\WW$-flabbiness of $X$. The flabbiness
condition means that we have a map
\[
  \extension{\phi}{P}{f}
\]
such that the diagram
\[
  \begin{tikzcd}
    P \ar[dr,"f"'] \ar[rr]
    & & \One \ar[dl,"\extension{\phi}{P}{f}",dashed] \\
    & X
  \end{tikzcd}
\]
commutes for any proposition $P : \Omega_{\WW}$ and function $f : P \to X$, with
witness
\[
  \extends{\phi}{P}{f}.
\]

In order to extend $f'$ as in the diagram below, first notice that it is of the
form $\pair{f,g}$ with $f$ as in the previous diagram and
$g : \Pi (p : P), A (f\,p)$.
\begin{equation}\label{Sigma-extension-diagram}
  \begin{tikzcd}
    P \ar[dr,"{f' \eqqcolon \pair{f,g}}"'] \ar[rr]
    & & \One \ar[dl,"{(x,a)}",dashed] \\
    & \Sigma (x : X) , A\,x.
  \end{tikzcd}
\end{equation}

\begin{definition}[Compatibility with the flabbiness witness]%
  \label{compatibility-condition}
  We define the \emph{compatibility map} $\rhotwo{P}{f}$ as follows:
  \begin{align*}
    \rhotwo{P}{f}  : A (\extension{\phi}{P}{f}) & \to \Pi (p : P) , A (f \, p) \\
    a & \mapsto \lambda p . \transportt{A}{(\extendsf{\phi}{P}{f}{p})}{a}.
  \end{align*}
  Notice that \(\rhotwo{P}{f}\) also depends on \(A\) and \(\phi\), and when necessary
  we make this dependency explicit by writing these parameters as superscripts
  of \(\rho\).
\begin{enumerate}
\item
  We say that the family $A$ is \emph{compatible with the flabbiness witness}
  $\phi$ of~$X$ if the map $\rhotwo{P}{f}$ is an equivalence for every $P$
  and $f$.
  Notice that this compatibility condition is property.
\item (Weak) \emph{compatibility data} is given by a \emph{section} of the
  compatibility map $\rhotwo{P}{f}$ for each $P$ and $f$.
\end{enumerate}
\end{definition}

\begin{theorem}\label{Sigma-flabby}
  Compatibility data for $A$ with the $\WW$-flabbiness witness $\phi$ of $X$ gives
  $\WW$-flabbiness data for the type \(\Sigma (x : X) , A\,x\).
\end{theorem}
\begin{proof}
  Let the maps \(f\), \(g\) and \(f' = \pair{f,g}\) be as in
  \eqref{Sigma-extension-diagram}, and write \(\sigma\) for the assumed section
  of \(\rhotwo P f\).
  We extend~\(f'\) by setting \(x \coloneqq \extension{\phi}{P}{f}\) and
  \(a \coloneqq \sigma\,g\).
  This is indeed an extension of \(f'\), for if we have \(p : P\), then
  \(x = \extension \phi P f = f\,p\) as witnessed by \(\extendsf \phi P f p\),
  so that we obtain \((x,a) = (f\,p , g\,p)\), because \(\sigma\) is a section of
  \(\rhotwo P f\).
\end{proof}

If the type family \(A\) is proposition valued, then a simpler criterion is
available.

\begin{corollary}\label{subtype-is-flabby}
  If \(A\) is proposition valued and we have a (necessarily unique) map
  \[
    \sigma_{P,f} : \paren*{\Pi (p : P) , A (f \, p)} \to A
    (\extension{\phi}{P}{f}),
  \]
  for every \(P : \Omega_{\WW}\) and \(f : P \to X\), then the subtype
  \(\Sigma (x : X) , A x\) of \(X\) is \(\WW\)\nobreakdash-flabby.
\end{corollary}
\begin{proof}
  When \(A\) is proposition valued, \(\sigma_{P,f}\) is necessarily a section of
  \(\rhotwo{P}{f}\). Hence \(A\) has compatibility data, so the
  result follows from \cref{Sigma-flabby}.
\end{proof}

The following shows that the sufficient condition of \cref{Sigma-flabby} for the
injectivity of \(\Sigma\)-types is not necessary in general.
\begin{proposition}\label{index-not-injective}
    The index type of an injective \(\Sigma\)-type need not be injective. 
\end{proposition}
\begin{proof}
  The type of pointed types is injective and equivalent to the \(\Sigma\)-type
  \[\Sigma (I : \Inh) , \fst I\] of pointed inhabited types, and the type \(\Inh\)
  of inhabited types is not injective in general by
  \cref{type-of-inhabited-types-injective-iff-projective-props} below.
\end{proof}

The following lemma can be used for showing that a type of the form
\[\Sigma (x : X) , A_1\, x \times A_2\, x\] is flabby when we already
have compatibility data for the families \(A_1\) and \(A_2\).
\begin{lemma} \label{product-of-structures:2}
  If the families $A_1$ and $A_2$ both have compatibility data for $\phi$, then so does
  the family $x \mapsto A_1\, x \times A_2\, x$. Similarly, the compatibility condition is preserved by binary products.
\end{lemma}
\begin{proof}
  Suppose we are given a proposition \(P\) and a map \(f : P \to X\).
  Let us drop the subscripts and write \(\rho_1\)~and~\(\rho_2\) for the
  maps of \cref{compatibility-condition} with respect to the type families
  \(A_1\) and \(A_2\), and \(\sigma_1\) and \(\sigma_2\) for their assumed
  sections.
  Then
  \begin{align*}
    \sigma : \Pi(p : P),A(f\,p) &\to A_1 (\extension{\phi}{P}{f}) \times A_2(\extension{\phi}{P}{f}) \\
    \alpha &\mapsto
  (\sigma_1({\fst} \comp \alpha) ,
             \sigma_2({\snd} \comp \alpha))
  \end{align*}
  is a section (respectively two-sided inverse), as is easily verified.
\end{proof}

If want to show that a type of the form
\[
  \Sigma (x : X) , \Sigma (a : \funone{A}{x}) , \funtwo{B}{x}{a}
\]
is flabby, where the family \(B\) happens to be proposition-valued and the
type \(\Sigma (x : X) , \funone{A}{x}\) is already known to be flabby, this
can often be done directly using \cref{subtype-is-flabby} if we consider
the equivalent form
\[
  \Sigma \left(\sigma : (\Sigma x : X , \funone{A}{x})\right) , \funone{C}{\sigma}
\]
again with \(C\) proposition valued.
The following lemma strengthens both the hypothesis and the conclusion
of this example, by showing that the compatibility data is
preserved if \(B\) is closed under extension in a suitable sense.

\begin{lemma} \label{compatibility-with-axioms}
  The family \(x \mapsto \Sigma (a : \funone{A}{x}) , \funtwo{B}{x}{a}\) has
  compatibility data with \(\phi\) provided the family \(A\) does and the following properties hold:
  \begin{enumerate}
  \item the family \(B\) is proposition valued,
  \item the family \(B\) is \emph{closed under extension}, in the sense that for
    every \(P : \Omega_{\WW}\), \(f : P \to X\) and
    \(\alpha : \Pi (p : P) , A (\funone{f}{p})\), if
    \[
      \funtwo{B}{(\funone{f}{p})}{(\funone{\alpha}{p})} \text{ holds for every }
      p : P,
    \]
    then
    \[
      \funtwo{B}{(\extension{\phi}{P}{f})}{(\funone{\sigma_{\rho,f}}{\alpha})}
      \text{ holds},
    \]
    where \(\sigma_{\rho,f}\) is the section of \(\rhotwo{P}{f}\) given by the
    compatibility data of~\(A\).
  \end{enumerate}
  Moreover, the family \(x \mapsto \Sigma (a : \funone{A}{x}) , \funtwo{B}{x}{a}\) satisfies the compatibility condition if \(A\) does and the above two properties hold.
\end{lemma}
\begin{proof}
  Write \(A'\) for the type family on \(X\) given by
  \(\funone{A'}{x} \coloneqq \Sigma (a : \funone{A}{x}) , \funtwo{B}{x}{a}\) and
  let \(P : \Omega_{\WW}\) and \(f : P \to X\).
  We need to show that the map
  \begin{align*}
    &\rho' : A' (\extension{\phi}{P}{f}) \to \Pi (p : P) , A' (f \, p) \\
    &\rho' \coloneqq \rhosuper{P}{f}{A'}{\phi}
  \end{align*}
  has a section.
  To this end, we define
  \begin{align*}
    &\sigma' : (\Pi (p : P) , A' (f \, p)) \to A' (\extension{\phi}{P}{f}) \\
    &\sigma' \, \alpha \coloneqq \paren*{\sigma_{\rho,f}(\lambda p . \fst (\alpha \, p)) , \tau},
  \end{align*}
  where
  \(\tau : \funtwo{B}{(\extension{\phi}{P}{f})}{(\sigma_{\rho,f}\,\alpha)}\) is
  provided by the fact that \(B\) is closed under extension.

  To see that \(\sigma'\) is a section of \(\rho'\), we note that by function
  extensionality, it suffices to prove that
  \[
    \rho' \, (\sigma' \alpha) \, p = \alpha \, p
  \]
  for all \(p : P\) and \(\alpha : \Pi (p : P),A'(f\, p)\).
  Using the shorthands \(\alpha_1 \coloneqq \lambda p . \fst (\alpha \, p)\),
  \(\alpha_2 \coloneqq \lambda p . \snd (\alpha \, p)\) and
  \(e \coloneqq \extends{\phi}{P}{f}\) , we observe that:
  \begin{align*}
    \rho' (\sigma' \, \alpha) \, p
    &\equiv \rho' (\sigma \, \alpha_1 \, \tau) \, p \\
    &\equiv \transportt{A'}{(e \, p)}{(\sigma \, \alpha_1 , \tau)} \\
    &= (\transportt{A}{(e \, p)}{(\sigma \, \alpha_1)} , \tau') \\
    &\equiv (\rhosuper{P}{f}{A}{\phi}\,(\sigma_{P,f}\,\alpha_1) , \tau') \\
    &= (\alpha_1 \, p , \alpha_2 \, p) \\
    &\equiv \alpha \, p,
  \end{align*}
  where
  \(\tau' \coloneqq \transportd{A}{B}{(e \, p)}{(\sigma\,\alpha_1)}{\tau}\).
  The first non-definitional equality is an instance of \cref{transport:sigma},
  and the second holds because \(\sigma_{P,f}\) is a section of
  \(\rhotwo{P}{f}\) and because \(B\) is proposition-valued.
  Similarly, if the compatibility condition in the hypothesis is satisfied, $\sigma'$ is easily shown to be a retraction of $\rho'$ as well.
\end{proof}

\subsection{Mathematical structures revisited}
\label{sec:injectivity-mathematical-structures:more:general}

An example for which the techniques of~\cref{sec:injectivity-mathematical-structures} cannot be applied is the type of metric spaces, as metric spaces are not closed under propositionally indexed $\Pi$ in general. But they are closed under propositionally indexed $\Sigma$, as we show below, and hence, using the techniques of~\cref{injectivity-of-sigma-types}, we conclude that the type of metric spaces is injective.

The main work in this section on top of Section~\ref{injectivity-of-sigma-types}
is to make it easier to work with compatibility data in the case the index type
of a $\Sigma$ type is a universe, just as we did for closure under
propositionally indexed~$\Pi$ in~\cref{sec:injectivity-mathematical-structures}.

~ 

In this subsection we work with $S$, $T$ and $r$ as assumed in~\cref{sec:injectivity-mathematical-structures}.

\begin{definition}
We introduce a name for the canonical map induced by the $\Sigma$-flabbiness structure on the universe $\UU$:
\begin{align*}
 \rhosigma : \Pi (P : \Omega_\UU) , \Pi (A : P \to \UU) , (S ({\Sigma} A) \to \Pi (p : P) , S (A \, p)).
\end{align*}
We define this map using the $\Sigma$-injection defined in~\cref{piproj:sigmainj}, which is
an equivalence as $P$ is a proposition, where we omit $P$ and $A$ here and elsewhere:
\begin{align*}
  \rhosigma\,s\,p & \coloneqq T\,\sigmainj_p \,s.
\end{align*}
\end{definition}
\begin{definition}
\emph{Compatibility data for $\Sigma$} is a distinguished section for the map \[\rhosigma : S ({\Sigma} A) \to \Pi (p : P) , S (A \, p).\]
\end{definition}
\begin{lemma} \label{compatibility-gives-compatibility} Compatibility
  data for $\Sigma$ gives compatibility data for the canonical
  $\Sigma$-flabbiness data of the universe $\UU$.
\end{lemma}
\begin{proof}
  Assume a proposition $P : \Omega_\UU$ and a family $A : P \to \UU$. By
  equivalence induction, we conclude that $\rhosigma = \rhotwo{P}{A}$, where
  $\rho$ is as in~\cref{compatibility-condition} with $\phi$ taken to be the
  $\Sigma$-flabbiness data of~$\UU$.  Hence if $\Sigma$ has compatibility data,
  which amounts to saying that $\rhotwo{P}{A}$ has a distinguished section, then
  $\rhosigma$ has a section, as required.
\end{proof}

\begin{theorem}\label{type-of-structured-types-flabby}
  Compatibility data for $\Sigma$ gives
  flabbiness data for the type \[\Sigma (X : \UU) , S\,X.\]
\end{theorem}
\begin{proof}
  This follows directly by \cref{Sigma-flabby} and \cref{compatibility-gives-compatibility}.
\end{proof}

Notice that we can proceed in the same way for compatibility data for $\Pi$ and recover the results of~\cref{sec:injectivity-mathematical-structures}, but we feel that the present treatment, which is increasingly general, is more instructive and better motivated. In the formalization we include both the approach described here and this more general one.

The first of the following examples generalizes item~(\ref{exa:graphs}) of
\cref{examples:of:flabby:mathematical:structures}.
\begin{examples} \label{more:examples:of:flabby:mathematical:structures}
  The following are flabby. 
  \begin{enumerate}
  \item The type of \(R\)-valued graphs in a universe $\UU$ for an arbitrary
    type \(R : \VV\).
  \item The type of metric spaces in a universe $\UU$.
  \end{enumerate}
\end{examples}
\begin{proof}
  \emph{The type of graphs.} We consider \(S\,X \coloneqq (X \to X \to R)\) and
  choose
  \begin{align*}
    T & : X \simeq Y \to S\,X \to S\,Y
  \end{align*}
  just like in the proof of item (\ref{exa:graphs}) of
  \cref{examples:of:flabby:mathematical:structures}, i.e.\ we map an
  equivalence~\(f\) and a relation \(r\) to
  \(y \mapsto y' \mapsto r\,(f^{-1}\,y)\,(f^{-1}\,y')\).
  To get the injectivity of $\Sigma S$,
  \cref{type-of-structured-types-flabby} requires a section of
  $\rho_{\Sigma}$ for $P : \Omega_{\UU}$ and $A : P \to \UU$.
  We define the required section by
  \begin{align*}
    \sigma : (\Pi (p : P) , S(A\,p)) & \to S (\Sigma A) \\
    g & \mapsto \lambda (p,a) \, (p',a') \, . \, g\,p\,(\sigmainj_p^{-1}(p,a))\,(\sigmainj_p^{-1}(p',a')).
  \end{align*}
  It is straightforward to check that this is a section of $\rho_\Sigma$ using
  \cref{piproj:sigmainj} and the fact that transports over propositions are trivial.

  \emph{The type of metric spaces.}  By \cref{compatibility-with-axioms} it
  suffices to show that if we have, for every \(p : P\), a metric \(d_p\) on
  \(A\,p\), then \(\sigma\,(p \mapsto d_p)\), with \(\sigma\) as defined above,
  is a metric on \(\Sigma A\).
  This is straightforward, after one has proved some basic lemmas for handling
  multiple transports over propositions.
\end{proof}

\subsection{Subtypes}

We now give necessary and sufficient conditions for the injectivity of a subtype of an injective type.

\begin{lemma}\label{retract-reformulation}
  For a type \(D : \UU\) and a proposition-valued family \(P : D \to \VV\), the
  canonical embedding \((\Sigma (d : D),P\,d) \hookrightarrow D\) has a
  retraction if and only if we have a designated \(f : D \to D\) such that
  \begin{enumerate}[(a)]
  \item\label{item:fixed:points:a}
    $P (f\,d)$ for all $d : D$, and
  \item\label{item:fixed:points:b}
    $d$ is a fixed point of $f$ for every $d : D$ with $P\,d$.
  \end{enumerate}
\end{lemma}
\begin{proof}
  Given a retraction \(r\) of the canonical embedding, the map
  \(f \coloneqq {\fst} \circ r\) satisfies \eqref{item:fixed:points:a} because
  \(r\) goes into the subtype and it satisfies \eqref{item:fixed:points:b} as
  \(r\) is a retraction of \(\fst\).
  Conversely, given such a map \(f\) with \(g : \Pi(d : D),P(f\,d)\), we see
  that \({r\,d} \coloneqq (f\,d , g\,d)\) defines a retraction of the canonical
  embedding thanks to~\eqref{item:fixed:points:b}.
\end{proof}

\begin{theorem}\label{theorem:injective:subtype}%
  The following are logically equivalent for any
  $(\VV \sqcup \WW),\TT$\nobreakdash-injective type $D : \UU$ and any proposition-valued
  $P : D \to \VV$.
  \begin{enumerate}
  \item\label{item:subtype:injective}
    The subtype $\Sigma (d : D), P\,d$ is $(\VV \sqcup \WW),\TT$-injective.
  \item\label{item:subtype:retract}
    The subtype \(\Sigma (d : D), P\,d\) is a retract of \(D\).
   \item\label{item:canonical:retract} The canonical embedding
     \((\Sigma (d : D), P\,d) \hookrightarrow D\) has a retraction.
   \item\label{item:fixed:points}
     There is a designated $f : D \to D$ such that conditions (\ref{item:fixed:points:a}) and (\ref{item:fixed:points:b}) of~\cref{retract-reformulation} hold.
  \end{enumerate}
\end{theorem}

Modulo universe levels, the above theorem may be crudely summarized by saying that the
injective subtypes of an injective type are precisely its retracts.

\begin{proof}
  Items \(\eqref{item:canonical:retract}\) and \(\eqref{item:fixed:points}\) are equivalent by~\cref{retract-reformulation}.
  We show that
  \(\eqref{item:subtype:injective} \implies \eqref{item:canonical:retract} \implies
  \eqref{item:subtype:retract}\), while
  \(\eqref{item:subtype:retract} \implies \eqref{item:subtype:injective}\) is shown
  with more general universes in \cref{lemma:injective:subtype} below.
  The implication \(\eqref{item:canonical:retract} \implies
  \eqref{item:subtype:retract}\) is clear.

  \(\eqref{item:subtype:injective} \implies \eqref{item:canonical:retract}\).
  Since $P$ is valued in propositions, the first projection
  ${s : {(\Sigma (d : D), P\,d)} \to D}$ is an embedding, and hence, by
  \cref{embedding-retract}, the type $\Sigma (d : D), P\,d$ is a retract of $D$
  with $s$ as the section.
\end{proof}

In the following, notice that the universe $\VV$ is not mentioned in the injectivity conclusion, and so we get a more general result than the implication \(\eqref{item:subtype:retract} \implies \eqref{item:subtype:injective}\).
\begin{lemma}\label{lemma:injective:subtype}
  If $D$ is $\WW,\TT$-injective, $P : D \to \VV$ is proposition-valued, and one of the
  conditions~\eqref{item:subtype:retract}--\eqref{item:fixed:points} of~\cref{theorem:injective:subtype} holds,
  then the subtype $\Sigma (d : D), P\,d$ is $ \WW,\TT$-injective.
\end{lemma}
\begin{proof}
  It suffices to show this when
  condition \(\eqref{item:subtype:retract}\) holds, but then the
  injectivity follows immediately from~\cref{closure-under-retracts}.
\end{proof}

Because there are no small injective types unless \(\weakresizing{\UU}\) holds by~\cref{no-small-injectives} below, it is of interest to consider the particular case of the above where $\UU$ is a successor universe.
\begin{corollary} \label{corollary:injective:subtype}
  The following are logically equivalent for any $\UU$-injective type $D : \UU^+$ and any proposition-valued $P : D \to \UU$.
  \begin{enumerate}
  \item \label{corollary:injective:subtype:1}
        The subtype $\Sigma (d : D), P\,d$ is $\UU$-injective.
      \item \label{corollary:injective:subtype:2} One of the
  conditions~\eqref{item:subtype:retract}--\eqref{item:fixed:points} of~\cref{theorem:injective:subtype} holds with $\VV = \WW = \TT = \UU$ and
        $\UU$ instantiated to $\UU^{+}$.
  \end{enumerate}
\end{corollary}
The above characterization gives an alternative proof of the injectivity of
reflective subuniverses, which also follows from~\cite[Theorem~24]{Escardo2021}
and the fact that reflective subuniverses~\cite[Section~7.7]{HoTTBook} are
closed under products.
\begin{corollary}\label{reflective-subuniverse-injective}
  Any reflective subuniverse is injective.
\end{corollary}
\begin{proof}
  Given the data \(P : \UU \to \Omega\), \(\modality : \UU \to \UU\) and
  \(\eta : \Pi(A : \UU) , (A \to \modality A)\) of a reflective subuniverse, we
  recall that a type lies in the reflective subuniverse
  \(\UU_P \coloneqq \Sigma (A : \UU) , P\,A\) if and only if \(\eta_A\) is an
  equivalence. Hence, we can apply~\cref{corollary:injective:subtype} with
  \(f \coloneqq \modality\) in
        condition~\labelcref{item:fixed:points} of
        \cref{theorem:injective:subtype} to get the injectivity of~\(\UU_P\).
\end{proof}

\subsection{Models of generalized algebraic theories} \label{sec:gats:not:formalized}
In this section we give a uniform, but somewhat indirect, treatment for
establishing the injectivity of various types of \emph{set-level} mathematical
structures.
More precisely, we will show that the type of models of any generalized
algebraic theory (GAT) is injective.
We will not recall GATs here and instead refer to
Cartmell's~\cite{Cartmell1978,Cartmell1986} who introduced them. We have based
our development on \citeauthor{Frey2025}'s~\cite{Frey2025}.
As an example we recall from~\cite[(10.1)]{Frey2025} that the theory of
categories can be described as a GAT.
We also point out that we reserve the word \emph{category} for the usual notion
of a 1-category. In particular, we require the hom-types to be sets, but we do
not require the category to be univalent.

For our purposes it will be convenient to have an alternative presentation of
the type of models of a GAT.
By \cite[Example~2.8(c) and \S10]{Frey2025}, the category of models of a GAT is
equivalent to the category of models of a clan, the objects of which we recall
now.

\newcommand{\displaymap}{\rightarrowtriangle}
\newcommand{\clanT}{\mathbb{T}}

A \emph{clan}~\cite[Definition~2.1]{Frey2025}, as originally introduced by Taylor~\cite[\S4.3.2]{Taylor1987}, is a category
\(\clanT\) with a terminal object and a distinguished class of arrows called
\emph{display maps} containing all isomorphisms and terminal projections, closed
under composition, and such that for every display map
\(p : \Gamma' \displaymap \Gamma\) and any arbitrary map
\(s : \Delta \to \Gamma\), there exists a pullback
\[
  \begin{tikzcd}
    \Delta'
    \ar[dr,very near start,phantom,"\lrcorner"]
    \ar[r,"s'"] \ar[d,"q",-{Triangle[open]}]
    & \Gamma' \ar[d,"p",-{Triangle[open]}] \\
    \Delta \ar[r,"s"] & \Gamma
  \end{tikzcd}
\]
such that \(q\) is a display map.

We deviate slightly from~\cite{Frey2025} in that we do not require the category
\(\clanT\) to be small which is not important to us. (Note that Frey
considers large clans as well, it is just that clans are assumed to
be small by default.)

A \(\UU\)-\emph{model} (cf.~\cite[Definition~2.6]{Frey2025}) of a clan \(\clanT\) is a
functor into the category of sets in \(\UU\) that preserves the terminal object and
pullbacks of display maps.

\emph{Examples}: The theory of monoids, rings and categories can all be
described by GATs and the \(\UU\)-models of their syntactic clans~\cite[Example~2.8(c)
and \S10]{Frey2025} are precisely \(\UU\)-small monoids, rings and categories.

In the proof of the following, we can't apply the above results for
injectivity of mathematical structures, because a model of a clan is
not structure on a universe, but the proof uses the similar idea of
showing closure under proposition-indexed products.
\begin{theorem}\label{models-of-clan-injective}
  The type of \(\UU\)-models of any clan is \(\UU\)-injective.
\end{theorem}
\begin{proof}
  Let \(\clanT\) be an arbitrary clan and, in order to show that its type of \(\UU\)-models is flabby,  suppose we have a
  proposition \(P : \UU\) with for each \(p : P\) a model \(M_p\) of
  \(\clanT\).

  We define a  model \(M\) of \(\clanT\) as follows. At an object
  \(t\), we define
  \begin{align*}
  M\,t \coloneqq \Pi(p : P), M_p\,t.
  \end{align*}
  On morphisms, given an arrow \(a : t \to t'\) in \(\clanT\) we define \(M\,a\) by
  mapping \(\varphi : \Pi(p : P),M_p\,t\) to
  \(p \mapsto M_p\,a\,(\varphi\,p) : \Pi(p : P),M_p\,t'\), which is functorial
  thanks to the functoriality of \(M_p\) for each \(p : P\).

  We now show that \(M\) preserves the terminal object and pullbacks of display
  maps.
  For the terminal object \(1\), we use univalence and that \(M_p\) preserves
  the terminal object to get that
  \[
    {M\,1} \equiv ({\Pi(p : P),M_p\,1}) = ({\Pi(p : P),1}) = \One.
  \]

  For pullbacks of display maps, suppose we are given a pullback square
  \[
    \begin{tikzcd}
      t \times_s t' \ar[very near start,dr,"\lrcorner",phantom]
      \ar[r] \ar[d,-{Triangle[open]}]
      & t \ar[d,"a",-{Triangle[open]}] \\
      t' \ar[r,"b"] & s
    \end{tikzcd}
  \]
  in \(\clanT\) where the left and right map are display maps. Using univalence
  of \(\UU\) and for the category of sets in \(\UU\), and the fact that \(M_p\)
  preserves pullbacks of display maps for any \(p : P\), we get that
  \begin{align*}
    &{M\,(t \times_s t')} \\
    &\equiv
    \Pi(p : P), M_p\,(t \times_s t') \\
    &= \Pi(p : P), M_p\,t \times_{M_p\,s} M_p\,t' \\
    &= \Pi(p : P), \Sigma(X : M_p\,t),\Sigma(Y : M_p\,t'),(M_p\,a\,X = M_p\,b\,Y) \\
    &= \Sigma(\overline{X} : \Pi_{p : P}M_p\,t),
       \Sigma(\overline{Y} : \Pi_{p : P}M_p\,t'),
       \Pi(p : P),(M_p\,a\,(\overline{X}\,p) = M_p\,b\,(\overline{Y}\,p)) \\
    &= \Sigma(X : M\,t),\Sigma(Y : M\,t'),(M\,a\,X = M\,b\,Y) \\
    &= M\,t \times_{M\,s} M\,t',
  \end{align*}
  so \(M\) preserves pullbacks of display maps.
  Hence, \(M\) is a \(\UU\)-model of \(\clanT\).

  So far, we have not used that \(P\) is a proposition and it could have been an
  arbitrary type. Put differently, the above shows that the category of
  \(\UU\)-models of \(\clanT\) has (small) limits (cf.~\cite[Remark 2.9(a)]{Frey2025}).

  To finish the proof that the type of \(\UU\)-models of \(\clanT\) is
  \(\UU\)-flabby, we must show that \(M_p = M\) whenever we have \(p : P\).
  This is where we must use that \(P\) is a proposition.
  The
  preservation of the terminal object and
  pullbacks of display maps are
  properties, so by the structure identity
  principle, it suffices to show that the functors \(M_p\) and \(M\) are
  naturally isomorphic.
  At an object \(t\), the map
  \(\varepsilon_t : M\,t  \to M_p\,t\) that evaluates at~\(p\) is an equivalence by \cref{piproj:sigmainj}, recalling that $M\,t \equiv \Pi(x : P)M_x\,t$.
  Moreover, the maps \(\varepsilon_t\) assemble to a natural transformation,
  because for an arbitrary morphism \(a : t \to t'\) of \(\clanT\), the diagram
  \[
    \begin{tikzcd}
      \Pi(x : P)M_x\,t \ar[d,"\varepsilon_t"] \ar[r,"M\,a"] &
      \Pi(x : P)M_x\,t' \ar[d,"\varepsilon_{t'}"] \\
      M_p\,t \ar[r,"M_p\,a"] &
      M_p\,t'
    \end{tikzcd}
  \]
  commutes definitionally, by construction of \(M\,a\).
  We conclude that \(M_p = M\), which shows that the type of models is \(\UU\)-flabby.
\end{proof}

\begin{corollary}
  The type of \(\UU\)-small categories and the type of \(\UU\)-small rings are \(\UU\)-injective.
\end{corollary}

We emphasize that this section complements, but in no way subsumes, the treatment of
mathematical structures of
\cref{sec:injectivity-mathematical-structures,sec:injectivity-mathematical-structures:more:general}.
Indeed, \cref{models-of-clan-injective} only applies to \emph{set-level} structures,
whereas the results of the preceding sections also apply non-truncated
structures, such as \(\infty\)-magmas.
Another important difference is that the treatment of the previous sections is
directly applicable to a type of mathematical structures as usually defined in
type theory, whereas the treatment of this section requires us to present the
structures as models of a GAT instead.
Finally, unlike the previous sections, this section has not been formalized.

\section{Weak excluded middle and De Morgan's Law}\label{sec:WEM-DM}

A number of statements regarding injectivity, discussed in the following
sections, imply or are logically equivalent to the principle of weak excluded
middle.  It is well known that weak excluded middle is equivalent to De Morgan's
Law in topos logic~\cite{elephant}. There are various possibilities to formulate
the latter, some of which are propositions, and some of which are not
propositions in general, but in this section we show that they are all logically
equivalent (they imply each other) although not all of them are equivalent as
types.

Regarding foundational issues, we emphasize that although virtually
all results regarding injectivity rely on univalence, none of the
results in this section do, although some implications rely on
function extensionality (but not on propositional extensionality) and
on the existence of propositional truncations.

We start with a technical lemma that says that a type of the form $A + B$ is a proposition if and only if both $A$ and $B$ are propositions and $A$ and $B$ are together impossible.
\begin{lemma} \label{plus:prop}
  For any two types $A$ and $B$,
  \[
    \isprop (A + B) \iff \isprop A \times \isprop B \times \neg (A \times B).
  \]
It follows that if $P$ is any proposition then the types $P + \neg P$ and $\neg P + \neg \neg P$ are both propositions, equivalent to $P \vee \neg P$ and $\neg P \vee \neg \neg P$ respectively.
\end{lemma}
\begin{proof}
  $(\Rightarrow)$. We have that $A$ is a proposition because if $a,a' : A$ then $\inl a = \inl a'$ as $A + B$ is a proposition, and hence $a = a'$ as $\inl$ is always left-cancellable, and similarly $B$ is a proposition. Now let $a : A$ and $b : B$, then $\inl a = \inr b$ because $A + B$ is a proposition, and using the fact that we always have $\inl a \ne \inr b$, we get a contradiction.

  $(\Leftarrow)$. Given $x, y : A + B$, we consider four cases. If both are left elements, then they are equal because $A$ is a proposition. If both are right elements, then they are equal because $B$ is a proposition. If one is a left element and the other is a right element, we get a contradiction from the assumption that $\neg (A \times B)$, and we are done.

For the consequence of the implication $(\Rightarrow)$, notice that $\neg P$ is a proposition by function extensionality, and so we get that $P + \neg P$ is a proposition, as $P$ and $\neg P$ are together impossible, and hence $\proptrunc{P + \neg P} \implies P + \neg P$ by the elimination principle of propositional truncation, and similarly $\proptrunc{\neg P + \neg \neg P} \implies \neg P + \neg \neg P$.
\end{proof}

\begin{theorem}
  The following are logically equivalent.
  \begin{enumerate}
  \item \label{1:wem} \emph{The principle of weak excluded middle.}

    $\Pi (P : \UU), \isprop P \to \neg P + \neg\neg P$.

    (This is a proposition equivalent to $\Pi (P : \UU), \isprop P \to \neg P \vee \neg\neg P$.)

  \item \label{2:twem} \emph{The typal principle of weak excluded middle.}

    $\Pi (A : \UU), \neg A + \neg\neg A$.

    (This is a proposition equivalent to $\Pi (A : \UU), \neg A \vee \neg\neg A$.)

  \item \label{3:demorgan} \emph{De Morgan's Law}

    $\Pi (P, Q : \UU), \isprop P \to \isprop Q \to \neg (P \times Q) \to \neg P \vee \neg Q$.

    (This is a proposition.)

  \item \label{4:tdemorgan} \emph{Typal De Morgan's Law}

    $\Pi (A, B : \UU), \neg (A \times B) \to \neg A \vee \neg B$.

    (This is a proposition.)

  \item \label{5:udemorgan} \emph{Untruncated De Morgan's Law}

    $\Pi (P, Q : \UU), \isprop P \to \isprop Q \to \neg (P \times Q) \to \neg P + \neg Q$.

    (This is not a proposition in general.)

  \item \label{6:utdemorgan} \emph{Untruncated typal De Morgan's Law}

    $\Pi (A, B : \UU), \neg (A \times B) \to \neg A + \neg B$.

    (This is not a proposition in general.)

  \end{enumerate}
\end{theorem}
\begin{proof}
The types (\ref{1:wem})--(\ref{4:tdemorgan}) are propositions, using Lemma~\ref{plus:prop}, because propositions are closed under products, using function extensionality. That the types~(\ref{5:udemorgan}) and ~(\ref{6:utdemorgan}) are not propositions in general follows from Lemma~\ref{delta} below.
To establish the logical equivalences, we prove the following implications:
  
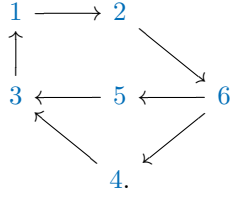
\[
    \begin{tikzcd}
     \ref{1:wem} \arrow[r]      & \ref{2:twem} \arrow[dr]                                                 \\
     \ref{3:demorgan} \arrow[u] & \ref{5:udemorgan} \arrow[l]   & \ref{6:utdemorgan} \arrow[l] \arrow[dl] \\
                                & \ref{4:tdemorgan}. \arrow[ul]
  \end{tikzcd}
  \]

 (\ref{1:wem})~$\implies$~(\ref{2:twem}). The negation of any type $A$ is a proposition by function extensionality, and hence, by the assumption of weak excluded middle with $P = \neg A$ we conclude that $\neg \neg A + \neg \neg \neg A$. Because three negations imply one, we in turn conclude that $\neg A + \neg \neg A$, as required.

 (\ref{2:twem})~$\implies$~(\ref{6:utdemorgan}). Using typal weak excluded middle on the type $A$ and then on the type $B$ in the case $\neg\neg A$, we in principle have the following three cases:
  \begin{enumerate}[(i)]
  \item $\neg A$ holds. Then $\neg A + \neg B$ holds and we are done.
  \item $\neg\neg A$ holds and $\neg B$ holds. Then again $\neg A + \neg B$ holds and we are done.
  \item $\neg\neg A$ holds and $\neg \neg B$ holds. But this case is impossible, as it contradicts the hypothesis $\neg(A \times B)$, because $\neg\neg A$ and $\neg \neg B$ together imply $\neg \neg(A \times B)$.
  \end{enumerate}

 (\ref{3:demorgan})~$\implies$~(\ref{1:wem}). We always have that $\neg (P \times \neg P)$ (known as the principle of non-contradiction). Then applying De Morgan we conclude that $\neg P \vee \neg\neg P$, from which we conclude that $\neg P + \neg\neg P$ by Lemma~\ref{plus:prop}, as required

  (\ref{4:tdemorgan})~$\implies$~(\ref{3:demorgan}). We apply typal De Morgan to the types $P$ and $Q$ ignoring the hypothesis that they are propositions.

  (\ref{5:udemorgan})~$\implies$~(\ref{3:demorgan}). We truncate the result given by untruncated De Morgan.

  (\ref{6:utdemorgan})~$\implies$~(\ref{4:tdemorgan}). We truncate the result given by untruncated typal De Morgan.

  (\ref{6:utdemorgan})~$\implies$~(\ref{5:udemorgan}). We apply untruncated typal De Morgan to the types $P$ and $Q$ ignoring the hypothesis that they are propositions.
\end{proof}

Of course, untruncated (typal) De Morgan is a proposition if it
doesn't hold, which is the case in some models. But if it does hold,
then it has at least two distinct witnesses, and hence isn't a
proposition.
\begin{lemma} \label{delta}
For every witness $\delta$ of untruncated (typal) De Morgan we can find a witness $\delta' \ne \delta$.
\end{lemma}
\begin{proof}
  For untruncated De Morgan, using $\delta$, we get a witness of weak excluded middle, and, in turn, using this, we
  define
  \[
    \delta' : \Pi (P, Q : \UU), \isprop P \to \isprop Q \to \neg (P \times Q) \to \neg P + \neg Q
  \]
  by
  \[
    \delta' \,P \, Q \, i \, j \, \nu =
    \begin{cases}
      \inr v & \text{if $\neg P$ and $v : \neg Q$ and $\delta \,P \, Q \, i \, j \, \nu$ is in left, } \\
      \inl u & \text{if $u : \neg P$ and $\neg Q$ and $\delta \,P \, Q \, i \, j \, \nu$ is in right, } \\
      \delta \,P \, Q \, i \, j \, \nu & \text{in all other cases,}
    \end{cases}
  \]
 using weak excluded middle on $P$ and on $Q$.
 Then clearly $\delta' \,\Zero \, \Zero \, i \, j \, \nu \ne \delta \,\Zero \, \Zero \, i \, j \, \nu$ for any $i,j : \isprop \Zero$ and $\nu : \neg (\Zero \times \Zero)$ and so $\delta' \ne \delta$.

 For untruncated typal De Morgan, we use the same argument, with the
 witnesses $i$ and $j$ deleted.
\end{proof}

\section{A Rice-like theorem for injective types}\label{sec:indecomposability}

In this section we establish a Rice-like theorem~\cite{Rice1953} for
injective types, which says that they have no nontrivial decidable
properties unless weak excluded middle holds.
Using decompositions as defined below, we may equivalently say that injective types are
indecomposable unless weak excluded middle holds.

The type of ordinals is a particular example of an injective type and
hence cannot have any nontrivial decidable property unless weak
excluded middle holds. Our result should be compared to
\cite[Theorem~63]{KrausNFXu2023}, which shows that the decidability of
certain properties of ordinals imply excluded middle.

\begin{definition}[Decomposition]
  For any type $X$, we define the type of decompositions of $X$ by
  \[
    \decomp X = \Sigma (f : X \to \Two) , \fiber_f(0) \times \fiber_f(1).
  \]
  We say that $X$ \emph{has a decomposition} if its type of decompositions is
  pointed.
\end{definition}

\begin{lemma}\label{decomposition-lemma}
  For any $X : \UU$, we have an equivalence
  \[
    (\Sigma (Y : \Two \to \UU) , (Y_0 + Y_1 \simeq X)) \,\simeq\, (X \to \Two).
  \]
\end{lemma}
\begin{proof}
  We have a chain of equivalences
  \begin{align*}
    \Sigma (Y : \Two \to \UU) , (Y_0 + Y_1 \simeq X)
    &\simeq \Sigma (Y : \Two \to \UU) , ((\Sigma (n : \Two) , Y_n) \simeq X) \\
    &\simeq \Sigma (Z : \UU) , (Z \to \Two) \times (Z \simeq X) \\
    &\simeq (X \to \Two),
  \end{align*}
  where the middle step follows from the fact that \(\UU\) is a type classifier,
  i.e.\ the type \(Z\) is obtained as the total space of \(Y\) with the
  projection map \(Z \to \Two\), and conversely, \(Y\) is obtained from \(Z\) by
  taking fibers.
\end{proof}

\begin{remark}
  For any type $X : \UU$,
  \[
    \decomp X \,\simeq\, (\Sigma (Y : \Two \to \UU) , (Y_0 + Y_1 \simeq X)
    \times Y_0 \times Y_1).
  \]
  Moreover, the type of decompositions of \(X\) is equivalent to the type of
  retracts of~\(X\) into~\(\Two\).
\end{remark}
\begin{proof}
  Under the equivalence that sends a map \(f : X \to \Two\) to its fibers
  \(Y : \Two \to \UU\), the types \(Y_0\) and \(Y_1\) correspond to
  \(\fiber_f(0)\) and \(\fiber_f(1)\), respectively. The first claim now follows
  from \cref{decomposition-lemma}.
  For the second claim, given \(f : X \to \Two\), we have equivalences
  \begin{align*}
    \fiber_f(0) \times \fiber_f(1)
    &\simeq
    \Pi (n : \Two) , \Sigma (x : X) , f \, x = n \\
    &\simeq \Sigma (s : \Two \to X) , \Pi (n : \Two) , f (s \, n) = n,
  \end{align*}
  so that the type of decompositions of $X$ is equivalent to the type of retracts of \(X\) into
  \(\Two\).
\end{proof}

\begin{proposition}\label{WEM-gives-decomposition-of-two-pointed-types}
  If weak excluded middle holds and the type $X$ has two distinct points, then
  $X$ has a decomposition.
\end{proposition}
\begin{proof}
  Suppose weak excluded middle holds and that \(X\) has points \(x_0 \neq x_1\).
  Using weak excluded middle, define \(f : X \to \Two\) by
  \[
    f\,x \coloneqq
    \begin{cases}
      0 &\text{if } \phantom{\lnot}\lnot (x \neq x_0), \\
      1 &\text{if } \lnot\lnot (x \neq x_0).
    \end{cases}
  \]
  Then \(f\,x_0 = 0\) and \(f\,x_1 = 1\), so \(f\) yields a decomposition of
  \(X\).
\end{proof}

\begin{definition}[\(\Omega_\UU\)-path]
  An \emph{\(\Omega_\UU\)-path in \(X\) from \(x : X\) to \(y : X\)} is a
  function \(f : \Omega_\UU \to X\) such that \(f \, \bot = x\) and
  \(f \, \top = y\).
  A type is said to \emph{have \(\Omega_\UU\)-paths} if we have an
  \(\Omega_\UU\)-path between any two points of the type.
\end{definition}

\begin{lemma}\label{decomposition-of-Omega-gives-WEM}
  From any decomposition of $\Omega_\UU$ we can derive weak excluded middle in
  the universe $\UU$.
\end{lemma}
\begin{proof}
  Suppose that \(f : \Omega_\UU \to \Two\) is a decomposition of \(\Omega_\UU\)
  with \(P_0,P_1 : \Omega_\UU\) such that \(f \, P_0 = 0\) and \(f \, P_1 = 1\).
  Given an arbitrary proposition \(Q\) in \(\UU\), we have to decide
  \(\lnot Q\).
  We consider the proposition
  \(R \coloneqq (P_0 \times Q) + (P_1 \times \lnot Q)\).
  If \(Q\) holds, then \(R = P_0\) by propositional extensionality, so that
  \(f\,R = f\,P_0 = 0\).  Similarly, if \(\lnot Q\) holds, then \(R = P_1\), so
  that \(f \, R = f \, P_1 = 1\).
  Taking contrapositives, we obtain the implications \(\lnot(f\,R = 0) \to \lnot Q\) and
  \(\lnot(f\,R=1) \to \lnot\lnot Q\).
  Since we have either \(f\,R = 0\) or \(f\,R = 1\), we get \(\lnot Q\) or
  \(\lnot\lnot Q\), as desired.
\end{proof}

\begin{lemma}\label{decomposition-of-type-with-Omega-paths-gives-WEM}
  If a type $X : \UU$ has a decomposition and $\Omega_\VV$-paths, then weak
  excluded middle in the universe~$\VV$ holds.
\end{lemma}
\begin{proof}
  Let \(f : X \to \Two\) be a decomposition of \(X\) with \(f \, x_0 = 0\) and
  \(f \, x_1 = 1\), and let \(g : \Omega_\VV \to X\) be an \(\Omega\)-path
  from \(x_0\) to \(x_1\).
  Because \(f (g \, \bot) = f \, x_0 = 0\) and \(f (g \, \top) = f \, x_1 = 1\),
  the composition \(f \circ g : \Omega_\VV \to X \to \Two\) is a decomposition
  of \(\Omega_\VV\) which implies weak excluded middle by
  \cref{decomposition-of-Omega-gives-WEM}.
\end{proof}

\begin{lemma}\label{Omega-paths-from-injectivity-naive}
  If a type $D:\UU$ is injective with respect to $\UU_0$ and $\WW^+$, then it
  has $\Omega_\WW$-paths.
\end{lemma}
\begin{proof}
  Let $x_0,x_1 : D$ and define $f : \Two \to D$ by $f \, i \coloneqq x_i$. Then by the
  injectivity of~$D$ we get an extension $g: \Omega_\WW \to D$ of $f$ along the
  embedding $\Two \hookrightarrow \Omega_\WW$ that maps $0$ to $\bot$ and $1$ to
  $\top$, and the desired $\Omega_\WW$-path is~$g$.
\end{proof}
We actually need a stronger version of this lemma with a lower universe. For
that purpose, we use \cref{ainjectivity-over-small-maps}.
\begin{lemma}\label{injectives-have-Omega-paths}
  If a type $D:\UU$ is injective with respect to $\VV$ and $\WW$, then it has
  $\Omega_\VV$-paths.
\end{lemma}
\begin{proof}
  The same as that of \cref{Omega-paths-from-injectivity-naive}, observing that
  the embedding $\Two \hookrightarrow \Omega_\VV$ is a $\VV$-small map
  (as its fiber at $P$ is the proposition $P + {\lnot P}$) and
  using \cref{ainjectivity-over-small-maps}.
\end{proof}

\begin{theorem}\label{decomposition-of-injective-type-gives-wem}
  If $D : \UU$ is injective with respect to $\VV$ and $\WW$, then from any
  decomposition of $D$ we get weak excluded middle in the universe~$\VV$.
\end{theorem}
\begin{proof}
  By \cref{injectives-have-Omega-paths,decomposition-of-type-with-Omega-paths-gives-WEM}.
\end{proof}

It follows that if any of the examples of injective types given
in~\cref{sec:examples} above has a decomposition, then weak excluded
middle holds in $\UU$.
Hence, the universe, the type of propositions, the type of ordinals,
the type of iterative (multi)sets, the types of \(\infty\)-magmas and
monoids, as well as carriers of sup-lattices and pointed dcpos are all
indecomposable unless weak excluded middle holds.

\begin{proposition}
  An injective type with two distinct points has an unspecified decomposition if and only if it has an explicit decomposition.
\end{proposition}
\begin{proof}
  For the nontrivial direction, suppose that \(D\) is an injective
   with two distinct points and has an unspecified decomposition. By
  \cref{WEM-gives-decomposition-of-two-pointed-types}, it suffices to derive
  WEM. But WEM is a proposition, so we may assume to have an explicit decomposition of \(D\), which gives WEM using
  \cref{decomposition-of-injective-type-gives-wem}.
\end{proof}

Notice that this proof is constructive but that it derives weak excluded middle as an intermediate step. An analogous situation is discussed by Andrej Bauer~\cite{bauer:on-complete-ordered-field:2019}.

\section{Counterexamples}\label{sec:counterexamples}

We discuss examples of types that are not injective in general. In
fact, the assumption that some types are injective implies that all
propositions are projective, that a weak form of the axiom of choice
holds, or that a weak form of excluded middle holds.  While in order
to establish examples of injective types we invariably rely on
univalence, in order to establish the counterexamples of this section,
functional and propositional extensionality suffice.

\subsection{The booleans}

The following result is also
observed in the context of topos theory~\cite[Proposition~D4.6.2]{elephant}.
\begin{counterexample}\label{Two-injective-iff-WEM}
  The type \(\Two\) is injective if and only if weak excluded middle holds.
\end{counterexample}
\begin{proof}
  Given a proposition \(P\), we consider the partial element
  \(f : P + \lnot P \to \Two\) given by \(f(\inl p) \coloneqq 1\) and
  \(f(\inr\nu) \coloneqq 0\). If \(\Two\) were injective, then \(f\) extends to
  an element \(x : \Two\). Now either \(x = 1\) or \(x = 0\). In the first case
  we obtain \(\lnot\lnot P\) and in the second case \(\lnot P\) must hold.
  Alternatively, since \(\Two\) decomposes as \(\One + \One\), we can appeal to
  the more general~\cref{decomposition-of-injective-type-gives-wem}.

  Conversely, we are going to use weak excluded middle to establish \(\Two\) as a
  retract of \(\Omega\). The result then follows as \(\Omega\) is injective and
  injectives are closed under retracts~\cite[Theorem~24 and
  Lemma~12]{Escardo2021}.
  The map \(s : \Two \to \Omega\) is the obvious inclusion of the booleans into
  the type of subsingletons.
  The map \(r : \Omega \to \Two\) sends \(P\) to \(0\) if \(\lnot P\) holds and
  maps to \(1\) if \(\lnot\lnot P\) holds using weak excluded middle. It is
  straightforward to check that \(s\) is a section of \(r\).
\end{proof}

\subsection{The conatural numbers}

We recall the type \(\NI\) of conatural numbers which may be implemented as the
type of decreasing binary sequences~\cite{Escardo2013}.
There is a canonical inclusion \(\underline{(-)} : \Nat \hookrightarrow \NI\) that maps a
number \(n\) to the sequence that starts with \(n\) ones and continues with all
zeroes.
We write \(\infty : \NI\) for the sequence with all ones.

We also recall that the weak limited principle of omniscience (WLPO) may be stated as:
\begin{equation}\label{WLPO}
  \Pi (u : \NI),(u = \infty) + (u \neq \infty). \tag{WLPO}
\end{equation}

\begin{lemma}\label{basic-discontinuity-taboo}
  The weak limited principle of omniscience holds if and only if we have a
  function \(f : \NI \to \NI\) with \(f\,\underline{n} = \underline{0}\) and
  \(f\,\infty = \underline{1}\).
\end{lemma}
\begin{proof}
  If WLPO holds, then we define \(f : \NI \to \NI\) by
  \(f\,u \coloneqq \underline 0\) if \(u \neq \infty\) and
  \(f\,u \coloneqq \underline 1\) if \(u = \infty\).
  Conversely, suppose we had such a function and let \(u : \NI\) be arbitrary.
  Notice that equality with \(\underline n\) is decidable for every \(n : \Nat\),
  so that we have \(f\,u = \underline 0\) or \(f\,u = \underline 1\).
  In the first case we have \(u \neq \infty\) as \(f\,\infty = \underline 1\),
  and in the second case we claim that \(u = \infty\). Indeed, it suffices to
  prove that \(u \neq \underline n\) for any \(n : \Nat\), but this follows as
  \(f\,{\underline n} = \underline 0\).
\end{proof}

\begin{counterexample}\leavevmode
  \begin{enumerate}
  \item\label{Ninf-injective-gives-WLPO} If the type \(\NI\) of conatural
    numbers is injective, then WLPO holds.
  \item\label{Ninf-injective-gives-WEM} More generally, if \(\NI\) is injective,
    then weak excluded middle holds.
  \end{enumerate}
\end{counterexample}
\begin{proof}
  \ref{Ninf-injective-gives-WLPO}: We consider the embedding
  \(j : {\Nat + \One} \hookrightarrow \NI\) that maps \(\inl n\) to
  \(\underline{n}\) and \(\inr \star\) to \(\infty\), and further consider the
  function \(g : {\Nat + \One} \to \NI\) that maps \(\inl n\) (for any
  \(n : \Nat\)) to \(\underline 0\) and \(\inr \star\) to \(\underline 1\).
  We denote the extension of \(g\) along \(j\) by \(f\).
  We claim that \(f\) satisfies the hypotheses of
  \cref{basic-discontinuity-taboo}.
  Indeed, \(f \, {\underline n} = f (j(\inl n)) = g(\inl n) = \underline{0}\) and
  \(f \, \infty = f(j(\inr \star)) = g (\inr \star) = \underline 1\).

  \ref{Ninf-injective-gives-WEM}: This follows from \cref{Two-injective-iff-WEM}
  as injective types are closed under retracts and \(\Two\) is a retract of
  \(\NI\) via the maps \(s : \Two \hookrightarrow \Nat \hookrightarrow \NI\) and
  \(r : \NI \to \Two\) given by \(r \, u \coloneqq u_0\).
\end{proof}

Even though the second result is more general (and has a simple proof), we
include the first result because it illustrates
(cf.~\cref{R-injective-taboos} below) that we can use hypothetical injectivity to
define discontinuous functions.

\subsection{The Dedekind reals}

\begin{counterexample}\label{R-injective-taboos}\leavevmode
  \begin{enumerate}
  \item\label{R-injective-gives-Heaviside-function} If the type \(\R\) of
    Dedekind real numbers is injective, then the discontinuous Heaviside
    function is definable.
  \item\label{R-injective-gives-WEM} Moreover, if \(\R\) is injective, then weak
    excluded middle holds.
  \end{enumerate}
\end{counterexample}
\begin{proof}
  \ref{R-injective-gives-Heaviside-function}: Recall that the Heaviside function
  \(H : \R \to \R\) is specified by \(H(x) = 0\) if \(x < 0\) and \(H(x) = 1\)
  if \(x \ge 0\).
  We consider the canonical embedding
  \[j : (\Sigma(x : \R),x<0) + (\Sigma(x : \R),x\ge 0) \hookrightarrow \R\] and
  the function \(h\) of the same type given by \(h(\inl x) = 0\) and
  \(h(\inr x) = 1\).
  Then the extension of \(h\) along \(j\) gives the Heaviside function.

  \ref{R-injective-gives-WEM}: Let \(P\) be a proposition and
  consider the partial element \(f : {P + \lnot P} \to \R\) that sends
  \(\inl p\) to \(0\) and \(\inr \nu\) to \(1\).
  If \(\R\) is injective, then we can extend \(f\) to a real number \(r : \R\)
  such that \(P \to r = 0\) and \(\lnot P \to r = 1\).
  By locatedness we have \(1/4 < r\) or \(r < 1/2\). In the first case,
  \(r \neq 0\), so that we get \(\lnot P\) and in the second case, \(r \neq 1\),
  so that we get \(\lnot\lnot P\).
  Hence, \(\lnot P\) is decidable.
\end{proof}

\subsection{Apartness spaces}

The above three counterexamples are instances of a general result
concerning types with nontrivial apartness relations.
There are various notions of apartness relation, which are to
be thought of as constructive strengthenings of the negation of equality, with
varying terminology, see e.g.~\cite[p.~7]{BridgesRichman1987},
\cite[p.~8]{MinesRichmanRuitenburg1988}, \cite[p.~8]{BridgesVita2011},
\cite[Definition~2.1]{BishopBridges1985} and
\cite[Section~8.1.2]{TroelstraVanDalen1988}.
Here we adopt the following definition: an \emph{apartness relation} \(\apart\)
on a type \(X\) is a binary proposition-valued relation on \(X\) that is
symmetric, irreflexive and \emph{cotransitive}:
\begin{center}
 \(x \apart y\) implies \(x \apart z\) or \(z \apart y\). 
\end{center}
\begin{definition}[Nontrivial apartness]
  We say that an apartness relation \(\apart\) on a type \(X\) is
  \emph{nontrivial} if we have \(x,y : X\) with \(x \apart y\).
\end{definition}

\begin{theorem}\label{injective-type-with-non-trivial-apartness-gives-WEM}
  If an injective type has a nontrivial apartness, then weak excluded middle
  holds.
  More precisely, if a type \(X : \UU\) is
  \(\TT,\WW\)-injective  and has a nontrivial apartness valued in \(\VV\), then
  weak excluded middle holds for the universe~\(\TT\).
\end{theorem}
Note that the parameters \(\UU\), \(\VV\) and \(\WW\) can vary freely in the above
theorem.
\begin{proof}
  Suppose \(X : \UU\) is a \(\TT,\WW\)-injective type with a nontrivial
  \(\VV\)-valued apartness relation, as witnessed by points \(x \apart y\), and
  let \(P : \TT\) be an arbitrary proposition.
  We consider the partial element \(f : P + \lnot P \to X\) given by
  \(f (\inl p) \coloneqq x\) and \(f (\inr \nu) \coloneqq y\).
  By the injectivity of \(X\), we can extend \(f\) to an element \(z : X\) such
  that \(P \to z = x\) and \(\lnot P \to z = y\).
  Taking contrapositives, we obtain \(z \neq x \to \lnot P\) and
  \(z \neq y \to \lnot\lnot P\).
  By assumption, \(x \apart y\), and therefore \(x \apart z\) or \(z \apart y\)
  by cotransitivity. But these respectively imply \(z \neq x\) and \(z \neq y\),
  so that we get \(\lnot P\) or \(\lnot\lnot P\).
\end{proof}

The converse holds in the following form and applies to non-injective types as
well, where we recall that a type in \(\UU^+\) is \emph{locally
  small}~(\cite[Definition~4.1]{Rijke2017}) if all of its identity types are
equivalent to types in \(\UU\).

\begin{theorem}\label{WEM-gives-that-type-with-two-distinct-points-has-nontrivial-apartness}\leavevmode
  \begin{enumerate}
  \item If weak excluded middle holds, then any type with two distinct points has a
  nontrivial apartness.
  More precisely, if weak excluded middle holds for the universe \(\UU\), then
  any type in \(\UU\) with two distinct points has a nontrivial \(\UU\)-valued
  apartness.
  \item More generally, if weak excluded middle holds for the universe \(\UU\), then
  any locally small type in \(\UU^+\) with two distinct points has a nontrivial
  \(\UU\)-valued apartness.
  \end{enumerate}
\end{theorem}
\begin{proof}
  We only spell out the proof of the more general statement. Suppose that weak
  excluded middle holds for \(\UU\) and that we have a locally small type
  \(X : \UU^+\) with distinct points~\(x_0\) and~\(x_1\).
  Since \(X\) is locally small, we have a \(\UU\)-valued identity type \(=_s\)
  of \(X\). We now define \(x \apart y \coloneqq \lnot(x =_s y)\).
  This relation is clearly irreflexive and symmetric, so it remains to prove
  cotransitivity.
  Suppose we have \(x \apart y\) and let~\(z : X\). We apply weak excluded middle to
  \(x \apart z\):
  \begin{itemize}
  \item If \(\lnot\lnot(x \apart z)\) holds, then we have \(\lnot(x =_s z)\),
    i.e.\ \(x \apart z\), since three consecutive negations can be reduced to a
    single negation.
  \item If \(\lnot(x \apart z)\) holds, then we claim that \(y \apart z\) holds.
    Indeed, unfolding the definition of \(y \apart z\), we see that it suffices
    to prove that \(y =_s z\) implies \(x \apart z\), but since \(x \apart y\),
    this follows by transport. \qedhere
  \end{itemize}
\end{proof}

In particular, type universes cannot have nontrivial apartness relations in
constructive type theory:

\begin{corollary}\label{non-trivial-apartness-on-universe-iff-WEM}
  A universe \(\UU\) has a \(\UU\)-valued nontrivial apartness if and
  only if weak excluded middle holds for \(\UU\).
\end{corollary}
\begin{proof}
  A universe \(\UU\) is \(\UU\)-injective, so that if we had a
  \(\UU\)-valued nontrivial apartness, then weak excluded middle holds for
  \(\UU\) via \cref{injective-type-with-non-trivial-apartness-gives-WEM}.
  Conversely, we use
  \cref{WEM-gives-that-type-with-two-distinct-points-has-nontrivial-apartness}
  together with the fact that universes are locally small: the
  identity type \(X = Y\) is equivalent (by univalence) to the small type of
  equivalences \(X \simeq Y\).
\end{proof}
This may
be seen as an internal version of \citeauthor{Kocsis2024}'s metatheoretic result
that Martin-L\"of Type Theory does not define
any non-trivial apartness relation on a universe~\cite[Corollary~5.7]{Kocsis2024}. In \emph{loc.\ cit.}\ this fact
is obtained by a parametricity argument.

\subsection{The simple types}

Recall that the \emph{simple types} are the smallest collection of
types closed under \(\Two\), \(\Nat\) and function types.
\begin{corollary}
  If any simple type is injective, then weak excluded middle holds.
\end{corollary}
\begin{proof}
  We inductively define nontrivial apartness relations on the simple types.
  Clearly, the negation of equality on the types \(\Two\) and \(\Nat\) is a
  nontrivial apartness relation on each of them.
  Now suppose that we have a nontrivial apartness relation \(\apart^Y\) on a
  (simple) type \(Y\) and consider the function type \(X \to Y\).
  Then the relation \(f \apart^{X \to Y} g\) on \(X \to Y\) given by
  \(\exists (x : X).\,f\,x \apart^Y g\,x\) is easily seen to be an apartness.
  Moreover, it is nontrivial for we assumed to have \(y_0 \apart^Y y_1\) so that
  \(\lambda x . y_0\) and \(\lambda x . y_1\) are apart in \(X \to Y\).
\end{proof}

\subsection{Tight apartness spaces}

Another class of counterexamples to injectivity is given by types that have a
tight, but not necessarily nontrivial, apartness relation.
We recall that an apartness \(\apart\) is \emph{tight} if two elements are equal
as soon as they are not apart, i.e.\ \(\lnot(x\apart y)\) implies \(x=y\).

\begin{theorem}
  Suppose \(X : \UU\) is a type with a tight apartness and not a subsingleton.
  If \(X\) is \(\TT,\WW\)-injective, then the double negation of the principle of weak excluded
  middle holds for \(\TT\).
\end{theorem}

Note that the universe parameter \(\WW\) can vary freely. Moreover, the tight
apartness can be \(\VV\)-valued for any universe \(\VV\).

\begin{proof}
  Let \(X\) be as in the statement of the theorem and assume for the
  sake of contradiction that weak excluded middle does not hold. We
  are going to show that~\(X\) must be a subsingleton, contradicting
  one of our assumptions, so let \(x,y : X\) be arbitrary.
  By tightness, it suffices to show that \(x\) and \(y\) are not apart.
  But if \(x\) and \(y\) were apart, then \(X\) has a nontrivial
  apartness, so that we can derive weak excluded middle via
  \cref{injective-type-with-non-trivial-apartness-gives-WEM},
  contradicting our assumption.
\end{proof}

A particular instance of the above theorem is given by the class of
\emph{totally separated} types~\cite{Escardo2019}. These types may be
characterized as precisely those types for which the \emph{boolean apartness} is
tight. For a type \(X\), this apartness, denoted by \(\apart_\Two\), is given by
\[
  x \apart_\Two y \coloneqq \exists (p : X \to \Two) .\,p\,x \neq p\,y.
\]
For comparison, we also give a more direct proof that nontrivial totally separated
types cannot be injective which exploits this particular apartness relation.

\begin{proposition}
  Suppose \(X : \UU\) is totally separated and not a subsingleton.
  If \(X\) is \(\TT,\WW\)-injective, then the double negation of weak excluded
  middle holds for \(\TT\).
\end{proposition}
\begin{proof}
  Suppose \(X\) is totally separated, injective and not a subsingleton. Using
  \cref{decomposition-of-injective-type-gives-wem} we reduce the claim to showing
  that it is false that \(X\) has no decomposition.
  Indeed, if \(X\) has no decomposition, then any map \(X \to \Two\) must be
  constant (in the sense that any of its values must agree).
  In particular, any two elements of \(X\) must be equal by tightness of the
  boolean apartness, contradicting our assumption that \(X\) is not a
  subsingleton.
\end{proof}

\subsection{Small types} \label{sec:small:types}

All the examples of injective types that we have presented so far, namely type universes, the type of small ordinals, the type of pointed types,
etc., are large collections of types.
This is no coincidence: any nontrivial injective type is large unless
the weak form of propositional resizing introduced
in~\cref{small-types-and-resizing} holds, as we now show.

\begin{theorem}\label{no-small-injectives}
  If we have a \((\UU \sqcup \VV),\WW\)-injective type \(D : \UU\) that has two distinct points~\(x_0\) and~\(x_1\), then \(\weakresizing{\UU}\) holds.
\end{theorem}
In the above theorem, the universe parameters \(\VV\) and \(\WW\) vary freely
and, for instance, may be taken to be the lowest universe \(\UU_0\).
A particular instance of the theorem is that we cannot have a \(\UU\)-injective
type that itself lives in \(\UU\) unless \(\Omeganotnot{\UU}\) is
\(\UU\)-small.

\begin{proof}
  Given such a type \(D\) and points \(x_0,x_1 : D\), we define
  \(f : \Two \to D\) as \(f \coloneqq [x_0,x_1]\).
  Further consider the embedding \(j : \Two \to \Omeganotnot{\UU}\) that sends
  \(0\) to \(\bot\) and \(1\) to \(\top\).
  This is a \(\UU\)-small map as its fiber at a \(\lnot\lnot\)-stable
  proposition \(P : \UU\) is the type \({P + \neg P} : \UU\).
  Hence, by \cref{ainjectivity-over-small-maps}, we get an extension
  \(s : \Omeganotnot{\UU} \to D\) of \(f\) along~\(j\), i.e.\ \(s \circ j = f\)
  holds.
  We claim that the map \(s\) has a retraction which completes the proof, as retracts
  of small types are small~\cite[Theorem~2.13]{deJongEscardo2023}.
  We define
  \begin{align*}
    r : D &\to \Omeganotnot{\UU} \\
    d &\mapsto d \neq x_0
  \end{align*}
  which is well defined as negated propositions are \(\neg\neg\)-stable.
  To see that \({r \circ s = \id}\), we claim that we have \(s\,P \neq x_0 \iff P\).
  Indeed, if \(P\) holds, then \(P = \top\), so
  \({s\,P = s(j\,1) = f\,1 = x_1 \neq x_0}\).
  For the converse we use that \(P\) is \(\neg\neg\)-stable, so that it suffices
  to prove \(\lnot P \Rightarrow {s\,P = x_0}\). But if \(\lnot P\) holds, then
  \(P = \bot\), so \(s\,P = s(j\,0) = f\,0 = x_0\), completing the proof.
\end{proof}

This internal fact is comparable to the metatheoretic result \cite[Corollary~10]{InjectivesCZF}, which states that CZF is consistent with the statement that the only injective \emph{sets} (as opposed to classes) are singletons, where the general uniformity principle is used for that purpose.

\subsection{The type of inhabited types}

We proved in \cref{type-of-nonempty-types-is-injective}  and Example~\ref{examples:of:flabby:mathematical:structures}(\ref{pointed}) that the type
\({\Sigma (X : \UU) , \neg\neg X}\) of non-empty types and the type  \(\Sigma (X : \UU) , X\) of pointed types are both injective.
We now show that the type
\[\Inh \coloneqq \Sigma (X : \UU) , \proptrunc{X}\] of inhabited types is not injective in general.
Bearing in mind the chain of implications $X \to \proptrunc{X} \to \neg\neg X$, this illustrates the constructive differences between propositional
truncation and double negation.

\begin{theorem}\label{type-of-inhabited-types-injective-iff-projective-props}
  The following are equivalent:
  \begin{enumerate}
  \item\label{Inh-injective} the type \(\Inh\) of inhabited types is injective;
  \item\label{Inh-retract} the type \(\Inh\) is a retract of \(\UU\);
  \item\label{projective-props} all propositions are projective: for every
    proposition \(P : \UU\) and type family \(Y : P \to \UU\), we have
    \[
      \pa*{\Pi (p : P) , \proptrunc{Y\,p}} \to
      \proptrunc*{\Pi (p : P) , Y\,p};
    \]
  \item\label{split-support} all types have unspecified split support, in the sense that
    for all \(X : \UU\), we have
    \[
      \proptrunc*{\,\proptrunc{X} \to X\,}.
    \]
  \end{enumerate}
\end{theorem}

The equivalence of \eqref{projective-props} and \eqref{split-support} was shown
in \cite[Theorem~7.7]{KrausEtAl2017}.
As also noted there, if \(Y\,p\) in \eqref{projective-props} is a two-element
set, then this is known as the \emph{world's simplest axiom of choice}, which fails
in some toposes~\cite{FourmanScedrov1982}.
Also notice that \eqref{projective-props}, and thus all the others, follow from
excluded middle.

\begin{proof}
  We prove
  \(\eqref{split-support} \Rightarrow \eqref{Inh-retract} \Rightarrow
  \eqref{Inh-injective} \Rightarrow \eqref{projective-props} \Rightarrow
  \eqref{split-support}\).

  \(\eqref{split-support} \Rightarrow \eqref{Inh-retract}\): By unspecified
  split support we have a map \({r : \UU \to \Inh}\) that sends a type
  \(X\) to the inhabited type \(\proptrunc{X} \to X\).
  When \(X\) is inhabited, we have
  \((\proptrunc{X} \to X) \simeq (\One \to X) \simeq X\), so by univalence, the
  first projection \(\Inh \to \UU\) is a section of \(r\).

  \(\eqref{Inh-retract} \Rightarrow \eqref{Inh-injective}\): This holds because
  injective types are closed under retracts.

  \(\eqref{Inh-injective} \Rightarrow \eqref{projective-props}\): Suppose that
  \(\Inh\) is injective and let \(P : \UU\) be a proposition with a type family
  \(Y : P \to \UU\) such that each \(Y\,p\) is inhabited.
  The map \(f : P \to \Inh\) takes \(p : P\) to the inhabited type \(Y\,p\).
  By flabbiness of \(\Inh\), we get an inhabited type \(X\) such that if we have
  \(p : P\) then \(X = Y\,p\).
  In particular, we have a map \(X \to \Pi (p : P) , Y\,p\), which proves the
  desired \(\proptrunc{\Pi(p : P), Y\,p}\) as \(X\) is inhabited.

  \(\eqref{projective-props} \Rightarrow \eqref{split-support}\): Given
  \(X : \UU\), simply take \(P \coloneqq \proptrunc{X}\) and
  \(Y\,p \coloneqq X\) in \eqref{projective-props}.
\end{proof}

\subsection{Infinite real projective space}

The implication \(\eqref{Inh-injective}\Rightarrow\eqref{split-support}\)
from~\cref{type-of-inhabited-types-injective-iff-projective-props} may be
generalized to the following result, which can then be instantiated to get a
non-injectivity result about the type of two-element types.

\begin{lemma}
  \label{family-has-unspecified-split-support-if-total-space-of-truncation-is-ainjective}%
  Let \(D : \UU\) be a \(\VV,\WW\)-injective type with a type family
  \(T : D \to \TT\).
  If the type \(\Sigma (d : D) , \proptrunc*{T\,d}\) is
  \((\TT \sqcup \VV'),\WW'\)-injective, then each \(T\,d\) has unspecified split
  support, i.e.\ \(\proptrunc*{\proptrunc*{T\,d} \to T\,d}\) for all \(d : D\).
\end{lemma}
Note that the parameters \(\VV'\) and \(\WW'\) can vary freely; in particular,
we can take them to be \(\UU_0\).
The above indeed generalizes the implication
\(\eqref{Inh-injective}\Rightarrow\eqref{split-support}\) from
\cref{type-of-inhabited-types-injective-iff-projective-props} by taking
\(D \coloneqq \UU\) and \(T : \UU \to \UU\) to be the identity.

\begin{proof}
  From the assumed injectivity of the type~\(\Sigma (d : D), \proptrunc*{T\,d}\),
  \cref{theorem:injective:subtype} yields a function \(f : D \to D\) such that
  for all \(x : D\) we have
  \begin{equation}\label{T-f-conditions}
    \proptrunc*{T(f\,x)}
    \text{ and }
    \proptrunc*{T\,x} \to {f\,x} = x.
  \end{equation}
  To prove \(\proptrunc*{\proptrunc*{T\,d} \to T\,d}\) for arbitrary
  \(d : D\), we first make use of the functoriality of propositional truncation to prove
  \(\proptrunc{T\,d} \to T\,d\) from \(T(f\,d)\).
  But given \(t : T(f\,d)\) and \(\tau : \proptrunc*{T\,d}\), we get
  \(f\,d = d\) via \eqref{T-f-conditions}, so that we can transport \(t\) along
  this equality to get an element of \(T\,d\), as desired.
\end{proof}

Recall that the type of two-element types is given by
\[
  \RPinf \coloneqq \sum_{X : \UU_0}\proptrunc*{X \simeq \Two}.
\]
The notation \(\RPinf\) comes from~\cite{BuchholtzRijke2017}, where
it is shown that this type has the universal property
of the infinite dimensional real projective space, defined as the sequential
colimit of the finite dimensional real projective spaces. In particular,
\(\RPinf\), which a priori lives in \(\UU_1\), is actually \(\UU_0\)-small in
the presence of such colimits.

We have already mentioned the world's simplest axiom of choice,
just after
\cref{type-of-inhabited-types-injective-iff-projective-props}, which
is \cref{projective-props} in that theorem but restricted to
families~\(Y\) where~\(Y\,p\) has precisely two elements.
For our purposes, an equivalent formulation will be convenient.

\begin{lemma} \label{wsac-charac}
  The world's simplest axiom of choice is equivalent to the
  assertion that for every \(X : \UU\), the type \(X \simeq \Two\) has
  unspecified split support.
\end{lemma}
\begin{proof}
  Suppose the world's simplest axiom of choice holds and let
  \(X : \UU\) be arbitrary.
  We set \(P \coloneqq \proptrunc*{X \simeq \Two}\) and let \(Y : P \to \UU\) be
  constantly \(X \simeq \Two\).
  Then \(X \simeq \Two\) has unspecified split support by the world's simplest
  axiom of choice, provided we show that each \(Y\,p\) has exactly two elements.
  So suppose we have \(p : P \equiv \proptrunc*{X \simeq \Two}\). Since we are
  proving a proposition, namely \(\proptrunc{Y\,p \simeq \Two}\), we may assume that \(X \simeq \Two\) in which case
  \(X \simeq \Two\) is equivalent to \(\Two \simeq \Two\) which indeed has
  exactly two elements.

  Conversely, suppose that for each \(X : \UU\), the type \(X \simeq \Two\) has
  unspecified split support.
  Let \(P : \UU\) be an arbitrary proposition and let \(Y : P \to \UU\) be such
  that \(Y\,p\) has exactly two elements for each \(p : P\).
  We have to prove \(\proptrunc*{\Pi\,Y}\).
  By assumption, the type \({\Pi\,Y} \simeq \Two\) has unspecified split support,
  so, using the functoriality of propositional truncation, it suffices to
  prove
  \[
    \pa*{\proptrunc*{{\Pi\,Y} \simeq \Two} \to {\Pi\,Y} \simeq \Two} \to \Pi\,Y.
  \]
  So assume \(h : \proptrunc*{{\Pi\,Y} \simeq \Two} \to {\Pi\,Y} \simeq \Two\) and
  \(p : P\). We are to give an element of \(Y\,p\).
  Since \(P\) is a proposition, we have \({\Pi\,Y} \simeq {Y\,p}\) by
  \cref{piproj:sigmainj}.
  Combined with \(h\), it now suffices to prove that
  \(\proptrunc*{Y\,p \simeq \Two}\) but this holds because \(Y\) is assumed to
  be a family of doubletons.
\end{proof}

\begin{theorem}
  If \(\RPinf\) is \(\VV,\WW\)-injective for given universes \(\VV\)
  and \(\WW\), then the world's simplest axiom of choice holds
  in the universe \(\UU_0\).
\end{theorem}
\begin{proof}
  Apply \cref{family-has-unspecified-split-support-if-total-space-of-truncation-is-ainjective,wsac-charac}
  with \(D \coloneqq \UU_0\) and \(T\,X \coloneqq (X \simeq \Two)\).
\end{proof}

Given that, in the presence of higher inductive types, $\RPinf$ is small, we may wonder whether we could instead apply the results of \cref{sec:small:types} to derive a taboo from its injectivity. However, $\RPinf$ is connected, and hence doesn't have two distinct points, so that \cref{no-small-injectives} is not applicable.

\section{Discussion}\label{sec:discussion}
After the non-injectivity of \(\RPinf\), it is natural to wonder about the
injectivity of higher homotopy types, other than the universe, such as the
circle.
Ulrik Buchholtz and David W\"arn have sketched models that invalidate the
injectivity of the circle~\cite{MathstodonDiscussion}. We leave it as an open
question whether one can derive a constructive or homotopical taboo under the
assumption that the circle is injective.

An anonymous reviewer of our TYPES abstract~\cite{InjectiveTYPES} suggested that
some of our results could be generalized to weak factorization
systems~\cite{Riehl2014}.
A~full investigation, in particular when it comes to the connections to
\emph{algebraic} weak factorization systems, is beyond the scope of this
paper. However, in our Agda development we have already constructed
factorizations and liftings where the left class is given by embeddings and the
right class is given by fiberwise injective maps, i.e.\ those maps whose fibers
are all injective
types~\cite{TypeTopologyWFS}.

\section{Acknowledgements}
We are grateful to Ulrik Buchholtz for his questions related to the injectivity
of models of theories, to Eric Finster for his comments on an earlier draft of
\cref{sec:gats:not:formalized}, to Andrew Swan for the argument given
in~\cref{small-types-and-resizing}, and to the anonymous referee for their
comments.
The first author was supported by The Royal Society [grant references
URF{\textbackslash}R1{\textbackslash}191055,
URF{\textbackslash}R{\textbackslash}241007].

\bibliography{references}

\end{document}